\documentclass[12pt,]{article}
\usepackage{mathrsfs}
\usepackage{amssymb,amsmath}
\usepackage{graphicx,epstopdf}
\usepackage{cases}
\usepackage{amsfonts}
\usepackage{color,xcolor,cite}
\usepackage{bm}
\usepackage[left=2.0cm,right=2.0cm,top=2.0cm,bottom=2.0cm]{geometry}
\usepackage[colorlinks,citecolor=blue,urlcolor=blue]{hyperref}

\numberwithin{equation}{section}
\newenvironment{proof}{{\noindent \it Proof.} }{\hfill $\square$ \par}
\newtheorem{theorem}{Theorem}[section]
\newtheorem{proposition}{Proposition}[section]
\newtheorem{corollary}{Corollary}[section]

\newtheorem{lemma}{Lemma}[section]

\newtheorem{remark}{Remark}[section]

\newcommand{\dd}{\mathrm{d}}
\newcommand{\bR}{\mathbb{R}}
\newcommand{\bE}{\mathbb{E}}
\newcommand{\bP}{\mathbb{P}}

\begin{document}
\title{Optimal State Equation for the Control of a Diffusion with Two Distinct  Dynamics}
	\author{Zengjing Chen$^{1}$, Panyu Wu$^{1,}$\footnote{Corresponding author. 
E-mail: wupanyu@sdu.edu.cn}~  and Xiaowen Zhou$^{2}$\\	
	{\small $^1$ Zhongtai Securities Institute for Financial Studies, Shandong University, Jinan 250100, China}\\
{\small $^2$ Department of Mathematics and Statistics, Concordia University, Montreal H3G 1M8, Canada}\\
	{\small (Emails: zjchen@sdu.edu.cn; wupanyu@sdu.edu.cn; xiaowen.zhou@concordia.ca)}}	
    \date{}
    \maketitle\noindent{}	
    {\bf Abstract:} We consider a class of stochastic control problems  which has been widely used in  optimal
foraging theory. The state processes have two distinct  dynamics, characterized by two pairs of drift and diffusion coefficients, depending on whether it takes values bigger or smaller than  a threshold value. Adopting a perturbation type approach,  we find an expression for potential measure of the optimal state process.
We then obtain an expression for the transition density of the optimal state process by inverting the associated Laplace transform.  Properties including the stationary distribution of the optimal state process are discussed. Finally, the expression of the value function is given for this class stochastic control problems.

\medskip


\medskip


	\section{Introduction}
\label{sec:intro}

In this paper, we revisit a stochastic  control problem  introduced in McNamara  \cite{McNamara} in which  a one-dimensional diffusion process is controlled by
switching between two given diffusions over  time interval $[0,T]$  with drift and diffusion coefficients $(\mu_1,\sigma_1)$ and $(\mu_2,\sigma_2)$, respectively. The objective is to maximize the probability that the process lies in $[0,\infty)$ at the final time $T$.
It is  known that the optimal state equation can be transformed into the following stochastic differential equation (SDE)
\begin{equation*}\label{sde0}
\dd X_t=(\mu_1 \mathbf{1}_{\{X_t \le 0\}}+\mu_2 \mathbf{1}_{\{X_t> 0\}})\dd t+(\sigma_1\mathbf{1}_{\{X_t\le 0\}}+\sigma_2\mathbf{1}_{\{X_t> 0\}})\dd B_t.
\end{equation*}
This control model was widely used to investigate the optimal diet of a forager faced with two types of preys, where
$(\mu_1,\sigma_1)$ and $(\mu_2,\sigma_2)$ denote the mean and standard deviation of the energy gain per minute under the two option, respectively. The animal can switch freely between the two options.
Here $X_t$ denotes the energy gained by capturing prey at time $t$ and the objective is to maximize 
% the chances of survival at the final time, i.e. 
the probability of $X_T\ge 0$; see, for example,  Houston and McNamara \cite{Houston,Houston2}, Kacelnik and  Bateson \cite{Kacelnik} and references therein.
By identifying an autonomous vehicle or software module as a forager
searching for tasks, Andrews et al.
\cite{Andrews} expanded this  foraging model to fit an autonomous vehicle control problem and provided how to make high-level control
decisions through this control model.

%Notice that, for this control problem,  although the optimal state equation or the optimal control is obtained, the value function is not expressed. In other words, it is  able to find the optimal foraging strategy.
But  the maximal survival probability  can not be evaluated.
To evaluate   the value function of this control problem,  one needs to  know the transition probability density of $X:=(X_t)_{t\ge 0}$.
Thus, it is an interesting problem and also the aim of this paper to characteristic the probability law of solution $X$ to %\eqref{sdea}
\begin{equation}\label{sdea}
\dd X_t=(\mu_1 \mathbf{1}_{\{X_t \le a\}}+\mu_2 \mathbf{1}_{\{X_t> a\}})\dd t+(\sigma_1\mathbf{1}_{X_t\le a}+\sigma_2\mathbf{1}_{\{X_t> a\}})\dd B_t,
\end{equation}
 where $a$ is any fixed real value.

In order to obtain the transition probability densities of SDE \eqref{sdea}, many scholars work on the following  two special cases SDEs \eqref{cltsde2} and  \eqref{cltsde1} for $\mu_1=\mu_2=\mu$ or $\sigma_1=\sigma_2=\sigma$  respectively,
\begin{equation}\label{cltsde2}
\dd X_t=\mu\dd t+(\sigma_1\mathbf{1}_{\{X_t\le a\}}+\sigma_2\mathbf{1}_{\{X_t> a\}})\dd B_t,
\end{equation}
\begin{equation}\label{cltsde1}
\dd X_t=(\mu_1 \mathbf{1}_{\{X_t \le a\}}+\mu_2 \mathbf{1}_{\{X_t> a\}})\dd t+\sigma\dd B_t.
\end{equation}
These two special cases SDEs \eqref{cltsde2} and  \eqref{cltsde1} are the optimal state equations for bang-bang controls of diffusion or drift, respectively.

For example, the transition probability density of $X$ satisfying \eqref{cltsde2} with bang-bang diffusion, 
%{\color{red}drift bang-bang diffusion $X$ satisfying \eqref{cltsde2},}   
known as an oscillating Brownian motion, was obtained by Keilson and Wellner  \cite{Keilson}  using the properties of occupation time and Laplace transform. McNamara  \cite{McNamara2} established a regularity condition on the distribution of $X$ for \eqref{cltsde2}.
The transition probability density of $X$ satisfying \eqref{cltsde1} can be recovered from the trivariate density of the Brownian motion, its local and occupation times obtained by Karatzas and Shreve  \cite{Karatzas}.  However, the expression for transition density of the diffusion process satisfying SDE \eqref{sdea} with different coefficients for both  drift and diffusion remains unknown.

Bene$\check{s}$
et al. \cite{Benes} first found an expression for transition density of a Brownian motion with a two-valued drift to characterize the
optimal state trajectory in a control problem. For the bang-bang control of drift, Steven \cite{Steven} obtained the explicit representation of
the optimal value function through the joint distribution of Brownian motion and its local time for square cost function.
Zhang  \cite{Zhang} derived a probabilistic representation of  the transition density of
Brownian motion with a general bounded piecewise-continuous drift function and applied it to a bang-bang control problem. 
Chen et al. \cite{CFLZ} gave an explicit representation of the optimal control and the optimal value
function for the bang-bang control problem of drift by the explicit solution of backward SDEs for symmetric cost function.
Ocone and Weerasinghe \cite{Ocone} obtained some explicit characterizations of the optimal control and optimal value functions with possibly degenerate variance control.
Decamps et al. \cite{Decamps} considered scalar diffusions with scale and speed functions
discontinuous at certain level and called that family of processes self exciting threshold diffusions and obtained semi-analytical expressions for the transition densities  of these diffusions. They also proposed several applications to interest rates modeling.
More recently, an optimal dividend problem is studied in Wang et al. \cite{WYYZ} for a risk process with endogenous regime switching whose  dynamics follows SDE (\ref{sdea}).

In this paper we aim at finding an expression of the value function for the control problem in McNamara  \cite{McNamara} by finding an expression for transition density of  the optimal state  process $X$ satisfying \eqref{sdea}. To this end, we adopt an approach that is different from those for the previous results along this line. More precisely, we first find an expression for potential measure $\mathbb{P}(X_{e_q}\in \dd z)$ for which we apply a perturbation argument and take use of solution to the exit problem for $X$. Here $e_q$ is an independent exponential random variable  with rate $q$. The desired expression of transition density then follows from Laplace inversion, which generalizes those results in  Karatzas and Shreve \cite{Karatzas}, as well as in Keilson and Wellner \cite{Keilson}.
Properties of the obtained transition density are also discussed. Specifically, we find that the optimal state equation is time reversible if and only if the diffusion and drift controls are both  single control. We also establish the continuity of the transition density function with respect to the initial state, and observe that the continuity with respect to the final state depends solely on  whether the diffusion control is single control, and is not unaffected by the drift controls. Moreover, we determine the stationary distribution of the optimal state process when $\mu_1>0$ and $\mu_2<0$. Lastly, we give the expression of the value function for the control problem.

The rest of the paper is structured as follows.
Section 2 contains preparations for the main results.
In Section 3, we derive an expression for the density of $X_{e_q}$.
% where $e_q$ is an independent exponential random variable with rate $q$.
In Section 4, we obtain  an expression for the transition density of $X$ and investigate  properties of the density function.
Examples and  graphs of the transition density function are provided in Section 5. Finally, in Section 6, we apply the findings to a control problem involving two given drift-diffusion pairs.
	
\section{Preliminaries}

Let $(\Omega,\mathcal{F},\mathbb{F},\mathbb{P})$ be a complete filtered probability space, where $\mathbb{F}=(\mathcal{F}_t)_{t\geq0}$ is the corresponding filtration satisfies the usual condition.
Fix any $a\in \mathbb{R}$ and $\mu_1,\mu_2\in\mathbb{R},\sigma_1,\sigma_2> 0$. Consider the following stochastic differential equation
\begin{eqnarray}\label{SDE}
\dd X_t=(\mu_1 \mathbf{1}_{\{X_t \le a\}}+\mu_2 \mathbf{1}_{\{X_t> a\}})\dd t+(\sigma_1\mathbf{1}_{\{X_t\le a\}}+\sigma_2\mathbf{1}_{\{X_t> a\}})\dd B_t
\end{eqnarray}
with two-valued drift and diffusion coefficients.  It follows from \'{E}tor\'{e} and  Martinez \cite{Etore} or Le Gall \cite{LeGall}  that \eqref{SDE} admits a unique strong solution.

The following solutions to the exit problems for process $X$ were first obtained in Wang et al. \cite{WYYZ}  using a martingale approach. We outline their derivations for convenience of the readers.

Write
\begin{eqnarray*}
	\label{theta12}
	\delta_1^{+}:=\frac{\sqrt{2q\sigma_1^2+\mu_1^2}+\mu_1}{\sigma_1^2},\quad 	\delta_2^{+}:=\frac{\sqrt{2q\sigma_2^2+\mu_2^2}+\mu_2}{\sigma_2^2},\\	 \delta_1^{-}:=\frac{\sqrt{2q\sigma_1^2+\mu_1^2}-\mu_1}{\sigma_1^2},\quad 	\delta_2^{-}:=\frac{\sqrt{2q\sigma_2^2+\mu_2^2}-\mu_2}{\sigma_2^2}.
\end{eqnarray*}
Then two solutions $g^{\pm}_{q}\in C^1(\bR)\cap C^2(\bR\backslash \{a\})$ to
\begin{eqnarray*}
	%\label{dis.generator.}
	\frac{1}{2}( \sigma_1^2 \mathbf{1}_{\{x\leq a\}}+\sigma_2^2 \mathbf{1}_{\{x>a\}})g^{\prime\prime}(x)+( \mu_1 \mathbf{1}_{\{x\leq a\}}+\mu_2 \mathbf{1}_{\{x>a\}})g^{\prime}(x)=qg(x)
\end{eqnarray*}
are
\begin{eqnarray}\label{gq}
	\begin{cases}
		g_{q}^-(x)
		=
	\left(c_-e^{\delta_1^-{(x-a)}}+(1-c_-) e^{-\delta_1^+{(x-a)}} \right)\mathbf{1}_{\{x\leq a\}}+	e^{-\delta_2^+{(x-a)}} \mathbf{1}_{\{x>a\}},\\
		g_{q}^+(x)
		=
	e^{\delta_1^-{(x-a)}}\mathbf{1}_{\{x\leq a\}}+	\left((1-c_{+})e^{\delta_2^-{(x-a)}}+c_+ e^{-\delta_2^+{(x-a)}} \right)\mathbf{1}_{\{x>a\}},
	\end{cases}
\end{eqnarray}
where the two constants $c_{\pm}$, determined by the smooth pasting conditions $g_{q}^{\pm\prime}(a-)=g_{q}^{\pm\prime}(a+)$, are given as
\begin{eqnarray*}
%\label{sign.c-}
c_-
=
\frac{\delta_1^+-\delta_2^+}{\delta_1^-+\delta_1^+},
\quad
1-c_-=\frac{\delta_1^-+\delta_2^+}{\delta_1^-+\delta_1^+}>0,
\end{eqnarray*}
and
\begin{eqnarray*}
%\label{sign.c+}
c_+=\frac{\delta_2^--\delta_1^-}{\delta_2^-+\delta_2^+},
\quad
1-c_+=\frac{\delta_2^++\delta_1^-}{\delta_2^-+\delta_2^+}>0.
\end{eqnarray*}
We introduce the following first hitting time of level $y\in\mathbb{R}$ for the process $X$ as
$$T_y:=\inf\{t\geq0;\,X_t=y\}$$
with the convention $\inf\emptyset=\infty$.
For $y\leq x\leq z$, applying the Meyer-It\^{o} formula (see Theorem 70 of
Chapter IV in Protter \cite{Protter}) we know that the two processes
$(e^{-qt}g_{q}^{\pm}(X_t))_{t\geq 0}$ are local martingales. Then, it holds that
\[g_{q}^-(x)=\bE_x [e^{-q(T_y\wedge T_z)}g_{q}^-(X_{T_y\wedge T_z})]
=g_{q}^-(y)\bE_x [e^{-qT_y}\mathbf{1}_{\{T_y<T_z\}}]+g_{q}^-(z)\bE_x [e^{-qT_z}\mathbf{1}_{\{T_z<T_y\}}], \]
and
\[g_{q}^+(x)=\bE_x [e^{-q(T_y\wedge T_z)}g_{q}^+(X_{T_y\wedge T_z})]
=g_{q}^+(y)\bE_x [e^{-qT_y}\mathbf{1}_{\{T_y<T_z\}}]+g_{q}^+(z)\bE_x [e^{-qT_z}\mathbf{1}_{\{T_z<T_y\}}]. \]
Solving the above two equations, we get the following solutions to the exit problems
\begin{eqnarray}\label{expectation1}
\bE_{x} [e^{-qT_y}\mathbf{1}_{\{T_y<T_z\}}]=\frac{g_{q}^+(z)g_{q}^-(x)-g_{q}^-(z)g_{q}^+(x)}{g_{q}^-(y)g_{q}^+(z)-g_{q}^-(z)g_{q}^+(y)},
\end{eqnarray}
and
\begin{eqnarray}\label{expectation2}
\bE_{x} [e^{-qT_z}\mathbf{1}_{\{T_z<T_y\}}]=\frac{g_{q}^+(y)g_{q}^-(x)-g_{q}^-(y)g_{q}^+(x)}{g_{q}^-(z)g_{q}^+(y)-g_{q}^-(y)g_{q}^+(z)}.
\end{eqnarray}
Combining \eqref{expectation1} and \eqref{gq}, we have for $y\le x$,
\begin{align}
\mathbb{E}_x\left[e^{-q T_{y}}\right]&=\mathbb{E}_x\left[e^{-q T_{y}}\mathbf{1}_{\{T_{y}<\infty\}}\right]=\lim_{z\to \infty}\mathbb{E}_x\left[e^{-q T_{y}}\mathbf{1}_{\{T_{y}<T_z\}}\right]=\frac{g_{q}^-(x)}{g_{q}^-(y)}.\label{expectation5}
\end{align}
Similarly, combining \eqref{expectation2} and \eqref{gq}, we have for $x\le z$,
\begin{align}
\mathbb{E}_x\left[e^{-q T_{z}}\right]=\frac{g_{q}^+(x)}{g_{q}^+(z)}.\label{expectation3}
\end{align}

Let $e_q$ be an exponential random variable with rate $q$ which is  independent of $X$. We proceed to find an expression for
\begin{equation}\label{laplace}
\bP_{x}(X_{e_q}\in \mathrm{d}z)=\int_0^\infty qe^{-qt}\bP_x(X_t\in \mathrm{d}z)\dd t
\end{equation}
using a perturbation approach and solution to the exit problem.
Then an expression for the transition  density function of $X$ can be obtained by inverting the above Laplace transform.

We recall some Laplace transforms for  linear diffusion process. The results in the following lemma can be found in  1.0.5 on p250 and in  2.0.1 on p295 of  Borodin and Salminen \cite{handbook}.
\begin{lemma}\label{linear}
For $ X_t^*=X_0^*+\mu  t+\sigma B_t$, write $T_z^*:=\inf\{t\ge 0;\ X_t^*=z\}$. For any $x,z\in \mathbb{R}$, we have
\begin{equation}\label{linearpeq}
\mathbb{P}_x(X_{e_q}^*\in\mathrm{d}z)=\frac{q}{\sqrt{2q\sigma^2+\mu^2}}\exp{\left(\frac{\mu(z-x)-|z-x|\sqrt{2q\sigma^2+\mu^2}}{\sigma^2}\right)}\mathrm{d}z,
\end{equation}
\begin{equation}\label{linearexp}
\mathbb{E}_xe^{-qT^*_z}=\exp{\left(\frac{\mu(z-x)-|z-x|\sqrt{2q\sigma^2+\mu^2}}{\sigma^2}\right)},
\end{equation}
\begin{equation}\label{linearlap}
\frac{1}{\sqrt{2q\sigma^2+\mu^2}}\exp{\left(\frac{\mu(z-x)-|z-x|\sqrt{2q\sigma^2+\mu^2}}{\sigma^2 }\right)} %(\mathrm{d}z+a)
=  \int_0^{\infty} e^{-q t} \frac{1}{\sqrt{2 \pi t\sigma^2}} \exp{\left(-\frac{(z-x-\mu t)^2}{2\sigma^2 t}\right)} \mathrm{d}t.%\quad \forall x,z\in\mathbb{R},
\end{equation}
\end{lemma}

Denote
$$ X_t^{(1)}=X_0^{(1)}+\mu_1 t+\sigma_1 B_t, \quad T_z^{(1)}:=\inf\{t\ge 0;\ X_t^{(1)}=z\},$$ and
$$ X_t^{(2)}=X_0^{(2)}+\mu_2 t+\sigma_2 B_t,\quad T_z^{(2)}:=\inf\{t\ge 0;\ X_t^{(2)}=z\}.$$

We will use the following function $h$ to represent the transition density of $X$ in Section 4, where
\begin{equation}\label{h}
h(t ; x, \mu)
 :=\frac{|x|}{\sqrt{2 \pi t^3}} \exp \left[-\frac{(x+\mu t)^2}{2 t}\right], \quad x, \mu \in \mathbb{R}, t>0,
\end{equation}
is introduced by Karlin and Taylor \cite{karlin} in Theorem 5.3 of Chapter 7 (see also p824 of Karatzas and Shreve \cite{Karatzas}).
It is easy to check that
\begin{equation}\label{hproperty1}
h(t ; x, \mu)=h(t ; -x, -\mu),\quad h(t ; -x, \mu)=h(t ; x, -\mu)=h(t ; x, \mu)e^{2\mu x}.
\end{equation}
Let $T_x^\mu $ be the passage time of a Brownian motion defined by
$$
T_x^\mu :=\inf \{t \geq 0 ; \mu t+B(t)=x\}.
$$
Then from the result on p197 of Karatzas and Shreve \cite{Karatzas2},
 %or  , we know that
\begin{align*}
\mathbb{P}_x\left(T_0^\mu \in d t\right)=h(t ; x,\mu) d t, \quad x, \mu \in \mathbb{R}, t>0, \\
\mathbb{P}_0\left(T_x^\mu \in d t\right)=h(t ; x, -\mu) d t, \quad x, \mu \in \mathbb{R}, t>0.
\end{align*}
In addition,
\begin{eqnarray}\label{hlap}
\int_0^{\infty}e^{-qt}h(t;x,\mu)\mathrm{d} t=
	\begin{cases}
		\exp{\left[-\left(\mu+\mathrm{sgn}(x)\sqrt{\mu^2+2q}\right)x\right]}, 
		   & x\neq0,\\
		0, & x=0,
	\end{cases}
\end{eqnarray}
and then
\begin{equation}\label{hproperty2}
	h(t ; x_1+x_2, \mu)=\int_0^{t}h(t-\tau ; x_1, \mu) h(\tau ; x_2, \mu)\mathrm{d} \tau\,\, \text{ for}\,\, x_1x_2>0.
\end{equation}

From now on we always write $X$ for the unique solution to SDE \eqref{SDE} and write $p(t; x, z)$ for its transition density from state $x$ to $z$ after time $t$.

\section{Probability distribution of $X_{e_q}$}\label{LT}
Recall that $e_q$ is an exponential random variable with rate $q$ that is  independent of $X$. In this section, we find expressions of $\bP_{x}(X_{e_q}\in \dd z)$ for different choices of $x$ and $z$.    %$X_{e_q}$.
\begin{theorem}\label{peqthm}
For any $x\ge a$, we have
\begin{eqnarray}\label{peq11}
\bP_{x}(X_{e_q}\in \dd z)=
	\begin{cases}
		\frac{qe^{-\delta_2^- (z-a)}}{\sqrt{2q\sigma_2^2+\mu_2^2}}\left[e^{\delta_2^- (x-a)}+\frac{(\delta_2^--\delta_1^-)e^{-\delta_2^+(x-a)}}{\delta_2^++\delta_1^-}\right]
\mathrm{d}z ,  \quad  z\ge x\ge a\\
		\frac{qe^{-\delta_2^+(x-a)}}{\sqrt{2q\sigma_2^2+\mu_2^2}}\left[e^{\delta_2^+ (z-a)}+\frac{(\delta_2^--\delta_1^-)e^{-\delta_2^-(z-a)}}{\delta_2^++\delta_1^-}\right]\mathrm{d}z , \quad x\ge z\ge a.
	\end{cases}
\end{eqnarray}
\end{theorem}
\begin{proof}
By the strong Markov property for $X$ and the
memoryless property for exponential random variable we have
\begin{align}
\mathbb{P}_{x}\left(X_{e_q} \in \mathrm{d}z \right)  =&\mathbb{P}_{x}\left(T_a<e_q, X_{e_q} \in \mathrm{d}z \right) +\mathbb{P}_{x}\left(T_a\ge e_q, X_{e_q} \in \mathrm{d}z \right)\nonumber\\
=& \mathbb{P}_{x}\left(T_a<e_q\right) \mathbb{P}_a\left(X_{e_q} \in \mathrm{d}z \right) + \mathbb{P}_{x}\left(T^{(2)}_a\ge e_q, X^{(2)}_{e_q} \in \mathrm{d}z \right)\nonumber\\
=&\mathbb{E}_{x} e^{-q T_a}\mathbb{P}_a\left(X_{e_q} \in \mathrm{d}z \right)+ \mathbb{P}_{x}\left(T^{(2)}_a\ge e_q, X^{(2)}_{e_q} \in \mathrm{d}z \right)\label{peq6}.
\end{align}
Since $x\ge a$, by \eqref{expectation5}, we have
\begin{equation}\label{expectation11}
\mathbb{E}_{x}e^{-qT_a}=\frac{g_q^-(x)}{g_q^-(a)}=e^{-\delta_2^+(x-a)}.
\end{equation}
For any $z\in \mathbb{R}$, let $r$ be any numbers between $a$ and $z$, we have
\begin{align*}
\mathbb{P}_{a}\left(X_{e_q} \in \mathrm{d}z \right)  =&\mathbb{P}_a\left(T_{r }<e_q, X_{e_q} \in \mathrm{d}z \right) \nonumber\\
=& \mathbb{P}_a\left(T_{r }<e_q\right) \mathbb{P}_{r }\left(X_{e_q} \in \mathrm{d}z \right) \nonumber\\
=&\mathbb{E}_{a} e^{-q T_{r }}\left[\mathbb{P}_{r }\left(e_q<T_{a}, X_{e_q} \in \mathrm{d}z \right)+\mathbb{P}_{r }\left(e_q>T_{a}, X_{e_q} \in \mathrm{d}z \right)\right]  \nonumber\\
=& \mathbb{E}_{a} e^{-q T_{r }}\left[\mathbb{P}_{r }\left(e_q<T_{a}, X_{e_q} \in \mathrm{d}z \right)+\mathbb{P}_{r }\left(e_q>T_{a}\right) \mathbb{P}_{a}\left(X_{e_q} \in \mathrm{d}z \right)\right]  \nonumber\\
=& \mathbb{E}_{a} e^{-q T_{r }}\left[\mathbb{P}_{r }\left(e_q<T_{a}, X_{e_q} \in \mathrm{d}z \right)+\mathbb{E}_{r } e^{-q T_{a}} \mathbb{P}_{a}\left(X_{e_q} \in \mathrm{d}z \right)\right].
\end{align*}
Thus,
\begin{equation}\label{peq0}
\mathbb{P}_{a}\left(X_{e_q} \in \mathrm{d}z \right)=\frac{\mathbb{E}_{a} e^{-q T_{r }} \mathbb{P}_{r }\left(e_q<T_{a}, X_{e_q} \in \mathrm{d}z \right)}{1-\mathbb{E}_{a} e^{-q T_{r }} \mathbb{E}_{r } e^{-q T_{a}}},\quad z\in\mathbb{R}.
\end{equation}

In this theorem $z\ge a$, thus, $a\le r\le z$. Combining \eqref{expectation5} and \eqref{expectation3}, we have
\begin{align}
\mathbb{E}_{a}e^{-qT_{r }}=\frac{g_q^+(a)}{g_q^+(r )}
%=\frac{1}{(1-c_+)e^{\delta_2^-(r-a)}+c_+e^{-\delta_2^+(r-a)}}
=\frac{\delta_2^++\delta_2^-}{(\delta_2^++\delta_1^-)e^{\delta_2^-(r-a)}+(\delta_2^--\delta_1^-)e^{-\delta_2^+(r-a)}},\label{expectation12}
\end{align}
\begin{align}
\mathbb{E}_{r }e^{-qT_{a}}=\frac{g_q^-(r )}{g_q^-(a)}=e^{-\delta_2^+(r-a)}.\label{expectation13}
\end{align}
Using the strong Markov property for $X^{(2)}$ and the
memoryless property for exponential random variable, we have
\begin{align}
\mathbb{P}_{r }\left(e_q<T_{a}, X_{e_q} \in \mathrm{d}z \right)=&\mathbb{P}_{r }\left(e_q<T^{(2)}_{a}, X^{(2)}_{e_q} \in \mathrm{d}z \right)\nonumber\\
=&\mathbb{P}_{r }\left(X^{(2)}_{e_q}\in \mathrm{d}z \right)-\mathbb{P}_{r }\left(e_q\ge T^{(2)}_{a},X^{(2)}_{e_q}\in \mathrm{d}z \right)\nonumber\\
=&\mathbb{P}_{r }\left(X^{(2)}_{e_q}\in \mathrm{d}z \right)-\mathbb{P}_{r }\left(e_q\ge T^{(2)}_{a}\right)\mathbb{P}_{a}\left(X^{(2)}_{e_q}\in \mathrm{d}z \right)\nonumber\\
=&\mathbb{P}_{r }\left(X^{(2)}_{e_q}\in \mathrm{d}z \right)-\mathbb{E}_{r }e^{-qT^{(2)}_{a}}\mathbb{P}_{a}\left(X^{(2)}_{e_q}\in \mathrm{d}z \right)\label{p11}
\end{align}
By Lemma \ref{linear}, we have
\begin{align}
&\mathbb{P}_{r }\left(e_q<T_{a}, X_{e_q} \in \mathrm{d}z \right)\nonumber\\
=&\frac{q}{\sqrt{2q\sigma_2^2+\mu_2^2}}\left[\exp\left(\frac{(z-r)}{\sigma_2^2}\left(\mu_2-\sqrt{2q\sigma_2^2+\mu_2^2}\right)\right)\right.\nonumber\\
 &\left.-\exp\left(\frac{-(r-a)}{\sigma_2^2}\left(\mu_2+\sqrt{2q\sigma_2^2+\mu_2^2}\right)\right)
\exp\left(\frac{(z-a)}{\sigma_2^2}\left(\mu_2-\sqrt{2q\sigma_2^2+\mu_2^2}\right)\right)\right]\mathrm{d}z \nonumber\\
=&\frac{q}{\sqrt{2q\sigma_2^2+\mu_2^2}}\left(e^{\delta_2^- (r-a)}-e^{-\delta_2^+ (r-a)} \right)e^{-\delta_2^- (z-a)} \mathrm{d}z .\label{p12}
\end{align}
Combing \eqref{peq0}, \eqref{expectation12}, \eqref{expectation13} and \eqref{p12}, we obtain that
\begin{align}
&\mathbb{P}_{a}\left(X_{e_q} \in \mathrm{d}z \right)\nonumber\\
 &=\frac{q}{\sqrt{2q\sigma_2^2+\mu_2^2}}e^{-\delta_2^- (z-a)} \frac{(\delta_2^++\delta_2^-) \left(e^{\delta_2^- (r-a)}-e^{-\delta_2^+ (r-a)} \right) \mathrm{d}z }{(\delta_2^++\delta_1^-)e^{\delta_2^-(r-a)}+(\delta_2^--\delta_1^-)e^{-\delta_2^+(r-a)}-(\delta_2^++\delta_2^-)e^{-\delta_2^+(r-a)}}\nonumber\\
 &=\frac{q}{\sqrt{2q\sigma_2^2+\mu_2^2}}\frac{ \delta_2^++\delta_2^-}{\delta_2^++\delta_1^-}e^{-\delta_2^- (z-a)}\mathrm{d}z. \label{p13}
\end{align}

Similarly to \eqref{p11} and by Lemma \ref{linear}, we have
\begin{align}
&\mathbb{P}_{x}\left(e_q<T^{(2)}_{a},X^{(2)}_{e_q}\in \mathrm{d}z \right)\nonumber\\
=&\mathbb{P}_{x}\left(X^{(2)}_{e_q}\in \mathrm{d}z \right)-\mathbb{E}_{x}e^{-qT^{(2)}_{a}}\mathbb{P}_{a}\left(X^{(2)}_{e_q}\in \mathrm{d}z \right)\nonumber\\
=&\frac{q}{\sqrt{2q\sigma_2^2+\mu_2^2}}\left[\exp\left(\frac{\mu_2(z-x)-|z-x|\sqrt{2q\sigma_2^2+\mu_2^2}}{\sigma_2^2}\right)\right.\nonumber\\
 &\left.-\exp\left(\frac{\mu_2(a-x)-(x-a)\sqrt{2q\sigma_2^2+\mu_2^2}}{\sigma_2^2}\right)
\exp\left(\frac{\mu_2(z-a)-(z-a)\sqrt{2q\sigma_2^2+\mu_2^2}}{\sigma_2^2}\right)\right]\mathrm{d}z \nonumber\\
=&\begin{cases}\frac{q}{\sqrt{2q\sigma_2^2+\mu_2^2}}\left(e^{\delta_2^- (x-a)}-e^{-\delta_2^+ (x-a)} \right)e^{-\delta_2^- (z-a)}\mathrm{d}z , \quad z\ge x\ge a,\\
\frac{q}{\sqrt{2q\sigma_2^2+\mu_2^2}}\left(e^{\delta_2^+ (z-a)}-e^{-\delta_2^- (z-a)} \right)e^{-\delta_2^+ (x-a)} \mathrm{d}z , \quad x\ge z\ge a.\\
\end{cases}\label{p14}
\end{align}

Putting together \eqref{peq6}, \eqref{expectation11}, \eqref{p13} and \eqref{p14}, we have
\begin{align}
&\mathbb{P}_{x}\left(X_{e_q} \in \mathrm{d}z \right)=\mathbb{E}_{x} e^{-q T_a}\mathbb{P}_a\left(X_{e_q} \in \mathrm{d}z \right)+ \mathbb{P}_{x}\left(T^{(2)}_a\ge e_q, X^{(2)}_{e_q} \in \mathrm{d}z \right)\nonumber\\
 =&e^{-\delta_2^+(x-a)} \frac{q }{\sqrt{2q\sigma_2^2+\mu_2^2}}\frac{\delta_2^++\delta_2^-}{\delta_2^++\delta_1^-}e^{-\delta_2^- (z-a)}\mathrm{d}z\nonumber\\
& +\begin{cases}\frac{q}{\sqrt{2q\sigma_2^2+\mu_2^2}}\left(e^{\delta_2^- (x-a)}-e^{-\delta_2^+ (x-a)} \right)e^{-\delta_2^- (z-a)}\mathrm{d}z, \  z\ge x\ge a, \nonumber\\
\frac{q}{\sqrt{2q\sigma_2^2+\mu_2^2}}\left(e^{\delta_2^+ (z-a)}-e^{-\delta_2^- (z-a)} \right)e^{-\delta_2^+ (x-a)} \mathrm{d}z, \ x\ge z\ge a,\nonumber\\
\end{cases}\nonumber\\
=&
\begin{cases}\frac{qe^{-\delta_2^-(z-a)}}{\sqrt{2q\sigma_2^2+\mu_2^2}}\left[e^{\delta_2^-(x-a)}+ \frac{(\delta_2^--\delta_1^-)e^{-\delta_2^+(x-a)}}{(\delta_2^++\delta_1^-)}\right]\mathrm{d}z , \quad z\ge x\ge a, \nonumber\\
	\frac{qe^{-\delta_2^+(x-a)}}{\sqrt{2q\sigma_2^2+\mu_2^2}}\left[e^{\delta_2^+(z-a)}+\frac{(\delta_2^--\delta_1^-)e^{-\delta_2^-(z-a)} }{(\delta_2^++\delta_1^-)}\right]\mathrm{d}z , \quad x\ge z\ge a.
\end{cases}\nonumber
\end{align}
We thus finish  the proof of Theorem \ref{peqthm}.
\end{proof}

\begin{theorem}\label{peqthm2}
	For any $x\ge a$, we have
		\begin{equation}\label{peq5}
		\mathbb{P}_{x}\left(X_{e_q} \in \mathrm{d}z \right)
		=\frac{q}{\sqrt{2q\sigma_1^2+\mu_1^2}}    \frac{\delta_1^++\delta_1^-}{\delta_2^++\delta_1^-} e^{-\delta_2^+(x-a)}e^{\delta_1^+(z-a)}  \mathrm{d}z  , \quad z<a.
	\end{equation}
\end{theorem}
\begin{proof}
By the strong Markov property for $X$ and the
memoryless property for exponential random variable we have
\begin{align}\label{peq3}
	\mathbb{P}_{x}\left(X_{e_q} \in \mathrm{d}z \right)  =\mathbb{P}_{x}\left(T_a<e_q, X_{e_q} \in \mathrm{d}z \right)
	&= \mathbb{P}_{x}\left(T_a<e_q\right) \mathbb{P}_a\left(X_{e_q} \in \mathrm{d}z \right)\nonumber\\
	&=\mathbb{E}_{x} e^{-q T_a}\mathbb{P}_a\left(X_{e_q} \in \mathrm{d}z \right).
\end{align}
Due to \eqref{peq0}, we have
\begin{equation}\label{peq4}
	\mathbb{P}_{a}\left(X_{e_q} \in \mathrm{d}z \right)=\frac{\mathbb{E}_{a} e^{-q T_{r }} \mathbb{P}_{r }\left(e_q<T_{a}, X_{e_q} \in \mathrm{d}z \right)}{1-\mathbb{E}_{a} e^{-q T_{r }} \mathbb{E}_{r } e^{-q T_{a}}},
\end{equation}
where $z<r<a$ in this proof.
Using the strong Markov property for $X^{(1)}$ and the
memoryless property for exponential random variable and Lemma \ref{linear}, we have
\begin{align}
	&\mathbb{P}_{r }\left(e_q<T_{a},X_{e_q}\in \mathrm{d}z \right)=\mathbb{P}_{r }\left(e_q<T^{(1)}_{a},X^{(1)}_{e_q}\in \mathrm{d}z \right)\nonumber\\
	=&\mathbb{P}_{r }\left(X^{(1)}_{e_q}\in \mathrm{d}z \right)-\mathbb{P}_{r }\left(e_q\ge T^{(1)}_{a},X^{(1)}_{e_q}\in \mathrm{d}z \right)\nonumber\\
		=&\mathbb{P}_{r }\left(X^{(1)}_{e_q}\in \mathrm{d}z \right)-
		\mathbb{P}_{r }\left(e_q\ge T^{(1)}_{a}\right)\mathbb{P}_{a}\left(X^{(1)}_{e_q}\in \mathrm{d}z \right)\nonumber\\
		=&\mathbb{P}_{r }\left(X^{(1)}_{e_q}\in \mathrm{d}z\right )-\mathbb{E}_{r }e^{-qT^{(1)}_{a}}\mathbb{P}_{a}\left(X^{(1)}_{e_q}\in \mathrm{d}z \right)\nonumber\\
		 =&\frac{q}{\sqrt{2q\sigma_1^2+\mu_1^2}}\left[\exp\left(\frac{(z-r)}{\sigma_1^2}\left(\mu_1+\sqrt{2q\sigma_1^2+\mu_1^2}\right)\right)\right.\nonumber\\
		 &\left.-\exp\left(\frac{(a-r)}{\sigma_1^2}(\mu_1-\sqrt{2q\sigma_1^2+\mu_1^2})\right)\exp\left(\frac{(z-a)}{\sigma_1^2}\left(\mu_1+\sqrt{2q\sigma_1^2+\mu_1^2}\right)\right)\right]\mathrm{d}z \nonumber\\
		=&\frac{q}{\sqrt{2q\sigma_1^2+\mu_1^2}}\left(e^{-\delta_1^+ (r-a)}-e^{\delta_1^- (r-a)} \right)e^{\delta_1^+ (z-a)} \mathrm{d}z . \label{5}
	\end{align}
	It follows from \eqref{expectation5} and \eqref{expectation3} that
	\begin{align}\label{expectation7}
		\mathbb{E}_{a}e^{-qT_{r }}=\frac{g_q^-(a)}{g_q^-(r )}
	=\frac{\delta_1^++\delta_1^-}{(\delta_1^+-\delta_2^+)e^{\delta_1^-(r-a)}+(\delta_2^++\delta_1^-)e^{-\delta_1^+(r-a)}},
	\end{align}
	\begin{equation}\label{expectation8}
		\mathbb{E}_{r }e^{-qT_a}=\frac{g_q^+(r )}{g_q^+(a)}=e^{\delta_1^-(r-a)}.
	\end{equation}
	Combine \eqref{expectation11} and \eqref{peq3}-\eqref{expectation8}, we have
	\begin{align*}
		&\mathbb{P}_{x}\left(X_{e_q} \in \mathrm{d}z \right)\\
=&\frac{q}{\sqrt{2q\sigma_1^2+\mu_1^2}}\frac{e^{-\delta_2^+(x-a)}e^{\delta_1^+(z-a)}(\delta_1^++\delta_1^-)(e^{-\delta_1^+(r-a)}-e^{\delta_1^-(r-a)})}
		{(\delta_1^+-\delta_2^+)e^{\delta_1^-(r-a)}+(\delta_2^++\delta_1^-)e^{-\delta_1^+(r-a)}-(\delta_1^++\delta_1^-)e^{\delta_1^-(r-a)}}  \mathrm{d}z \nonumber\\
=&\frac{q}{\sqrt{2q\sigma_1^2+\mu_1^2}}    \frac{\delta_1^++\delta_1^-}{\delta_2^++\delta_1^-} e^{-\delta_2^+(x-a)}e^{\delta_1^+(z-a)}  \mathrm{d}z  ,
	\end{align*}
	which proves the equality \eqref{peq5}.
\end{proof}

By the symmetry of Brownian motion, we have
$$\dd (-X_t)=(-\mu_2 \mathbf{1}_{\{-X_t < -a\}}-\mu_1 \mathbf{1}_{\{-X_t\ge -a\}})\dd t+(\sigma_2\mathbf{1}_{\{-X_t< -a\}}+\sigma_1\mathbf{1}_{\{-X_t\ge- a\}})\dd B_t,$$
Thus, we can obtain the following corollary from Theorem \ref{peqthm} and \ref{peqthm2}, and we omit its proof.
\begin{corollary}\label{peqthm3}
For any $x< a$, we have
\begin{eqnarray*}
\bP_{x}(X_{e_q}\in \dd z)=
	\begin{cases}
		\frac{qe^{\delta_1^+ (z-a)}}{\sqrt{2q\sigma_1^2+\mu_1^2}}\left[e^{-\delta_1^+ (x-a)}+\frac{( \delta_1^+-\delta_2^+)e^{\delta_1^-(x-a)}}{\delta_1^-+\delta_2^+}\right]
\mathrm{d}z,  \quad  z\le x< a\\
		\frac{qe^{\delta_1^-(x-a)}}{\sqrt{2q\sigma_1^2+\mu_1^2}}\left[e^{-\delta_1^- (z-a)}+\frac{(\delta_1^+ - \delta_2^+)e^{\delta_1^+(z-a)}}{\delta_1^-+\delta_2^+}\right]\mathrm{d}z, \quad x< z\le a\\
\frac{q}{\sqrt{2q\sigma_2^2+\mu_2^2}}    \frac{\delta_2^-+\delta_2^+}{\delta_1^-+\delta_2^+} e^{\delta_1^-(x-a)}e^{-\delta_2^-(z-a)}   \mathrm{d}z  , \quad\quad\quad\ \  z>a>x.
	\end{cases}
\end{eqnarray*}
\end{corollary}

\section{Transition density of the diffusion process}

Inverting the Laplace transforms obtained in Section \ref{LT}, we can find expressions for the transition density of $X$.

\begin{theorem}\label{densityth1}
%For diffusion process $X$ satisfying \eqref{SDE} and 
For any $x,z\ge a$, we have
\begin{align}\label{density1}
&\bP_{x}(X_t\in \dd z) \nonumber\\
=&\left\{\frac{2}{\sigma_2^2} \int_0^{\infty} \int_0^t e^{\frac{2\mu_2(z-a)}{\sigma_2^2}} h\left(t-\tau ; \frac{b}{\sigma_1}, -\frac{\mu_1}{\sigma_1}\right) h\left(\tau ;\frac{z+x-2a+b}{\sigma_2},\frac{ \mu_2}{\sigma_2}\right) \mathrm{d} \tau \mathrm{d} b\right.\nonumber\\
 &\left.+\frac{1}{\sqrt{2 \pi t\sigma_2^2}} \left[\exp \left(-\frac{\left(z-x-\mu_2t\right)^2}{2 t\sigma_2^2}\right)
-\exp \left(-\frac{\left(z+x-2a-\mu_2t\right)^2}{2 t\sigma_2^2}-\frac{2 \mu_2(x-a)}{\sigma_2^2}\right)\right]\right\}\mathrm{d}z.
\end{align}
\end{theorem}

\begin{proof}
Taking a Laplace transform on the convolution, by \eqref{hlap}  we obtain
\begin{align}
&\int_0^\infty qe^{-qt}\frac{2}{\sigma_2^2} \int_0^{\infty} \int_0^t e^{\frac{2\mu_2(z-a)}{\sigma_2^2}} h\left(t-\tau ; \frac{b}{\sigma_1}, -\frac{\mu_1}{\sigma_1}\right) h\left(\tau ;\frac{z+x-2a+b}{\sigma_2},\frac{ \mu_2}{\sigma_2}\right) \mathrm{d} \tau \mathrm{d} b \dd t\nonumber\\
=& \frac{2q}{\sigma_2^2} e^{\frac{2\mu_2(z-a)}{\sigma_2^2}} \int_0^\infty  \int_0^{\infty}e^{-qt} \int_0^t  h\left(t-\tau ; \frac{b}{\sigma_1}, -\frac{\mu_1}{\sigma_1}\right) h\left(\tau ;\frac{z+x-2a+b}{\sigma_2},\frac{ \mu_2}{\sigma_2}\right) \mathrm{d} \tau \dd t \mathrm{d} b \nonumber \\
=& \frac{2q}{\sigma_2^2} e^{\frac{2\mu_2(z-a)}{\sigma_2^2}} \int_0^\infty  \left[\int_0^{\infty}e^{-qt}   h\left(t; \frac{b}{\sigma_1}, -\frac{\mu_1}{\sigma_1}\right) \mathrm{d} t  \int_0^{\infty}e^{-qt}   h\left(t;\frac{ z+x-2a+b}{\sigma_2},\frac{ \mu_2}{\sigma_2}\right) \mathrm{d} t\right] \mathrm{d} b \nonumber\\
=& \frac{2q}{\sigma_2^2} e^{\frac{2\mu_2(z-a)}{\sigma_2^2}} \int_0^\infty \exp{\left[-\frac{\sqrt{2q\sigma_1^2+\mu_1^2}-\mu_1}{\sigma_1^2}b\right]}\exp{\left[-\frac{\sqrt{2q\sigma_2^2+\mu_2^2}+\mu_2}{\sigma_2^2}(z+x-2a+b)\right]} \mathrm{d} b \nonumber\\
=& \frac{2q}{\sigma_2^2} e^{\frac{2\mu_2(z-a)}{\sigma_2^2}} \exp{\left[-\frac{\sqrt{2q\sigma_2^2+\mu_2^2}+\mu_2}{\sigma_2^2}(z-a)\right]}    e^{-\delta_2^+ (x-a)} \int_0^\infty e^{-(\delta_2^+ +\delta_1^-)b}\mathrm{d} b \nonumber\\
=& \frac{2q}{\sigma_2^2(\delta_2^+ +\delta_1^-)} e^{-\delta_2^- (z-a)}e^{-\delta_2^+ (x-a)}.\label{part1}
\end{align}

For the case $z\ge x\ge a$, by \eqref{linearlap}, we can obtain
\begin{align}
&\int_0^\infty qe^{-qt}\frac{1}{\sqrt{2 \pi t\sigma_2^2}} \left[\exp \left(-\frac{\left(z-x-\mu_2t\right)^2}{2 t\sigma_2^2}\right)
-\exp \left(-\frac{\left(z+x-2a-\mu_2t\right)^2}{2 t\sigma_2^2}-\frac{2 \mu_2(x-a)}{\sigma_2^2}\right)\right] \mathrm{d}t \nonumber\\
=& \frac{q}{\sqrt{2q\sigma_2^2+\mu_2^2}}\left[\exp{\left(\frac{\left(\mu_2-\sqrt{2q\sigma_2^2+\mu_2^2}\right)(z-x)}{\sigma_2^2 }\right)}\right.\nonumber\\
&\qquad \qquad\qquad \ \left.-\exp{\left(\frac{\left(\mu_2-\sqrt{2q\sigma_2^2+\mu_2^2}\right)(z+x-2a)}{\sigma_2^2 }\right)}e^{-\frac{2\mu_2(x-a)}{\sigma_2^2}}\right]\nonumber\\
=&\frac{q}{\sqrt{2q\sigma_2^2+\mu_2^2}}\left(e^{\delta_2^-(x-a)}-e^{-\delta_2^+(x-a)}\right) e^{-\delta_2^-(z-a)}.\label{part2}
\end{align}
Combine \eqref{part1} and  \eqref{part2}, by \eqref{peq11}, we have
\begin{align}
&\int_0^\infty qe^{-qt}\left\{\frac{2}{\sigma_2^2} \int_0^{\infty} \int_0^t e^{\frac{2\mu_2(z-a)}{\sigma_2^2}} h\left(t-\tau ; \frac{b}{\sigma_1}, -\frac{\mu_1}{\sigma_1}\right) h\left(\tau ;\frac{z+x-2a+b}{\sigma_2},\frac{ \mu_2}{\sigma_2}\right) \mathrm{d} \tau \mathrm{d} b\right.\nonumber\\
 &\left.+\frac{1}{\sqrt{2 \pi t\sigma_2^2}} \left[\exp \left(-\frac{\left(z-x-\mu_2t\right)^2}{2 t\sigma_2^2}\right)
-\exp \left(-\frac{\left(z+x-2a-\mu_2t\right)^2}{2 t\sigma_2^2}-\frac{2 \mu_2(x-a)}{\sigma_2^2}\right)\right]\right\}\mathrm{d}t \mathrm{d}z\nonumber\\
=&\left[\frac{2q}{\sigma_2^2(\delta_2^+ +\delta_1^-)} e^{-\delta_2^- (z-a)}e^{-\delta_2^+ (x-a)}+\frac{q}{\sqrt{2q\sigma_2^2+\mu_2^2}}\left(e^{\delta_2^-(x-a)}-e^{-\delta_2^+(x-a)}\right) e^{-\delta_2^-(z-a)}\right]\mathrm{d}z\nonumber\\
=&	\frac{qe^{-\delta_2^- (z-a)}}{\sqrt{2q\sigma_2^2+\mu_2^2}}\left[e^{\delta_2^- (x-a)}+\frac{(\delta_2^--\delta_1^-)e^{-\delta_2^+(x-a)}}{\delta_2^++\delta_1^-}\right]
\mathrm{d}z\nonumber\\
=&\mathbb{P}_{x}\left(X_{e_q} \in \mathrm{d}z\right),  \quad  z\ge x\ge a.\label{peq01}
\end{align}

For the case $x\ge z\ge a$, by \eqref{linearlap} we obtain
\begin{align}
	&\int_0^\infty qe^{-qt}\frac{1}{\sqrt{2 \pi t\sigma_2^2}} \left[\exp \left(-\frac{\left(z-x-\mu_2t\right)^2}{2 t\sigma_2^2}\right)
	-\exp \left(-\frac{\left(z+x-2a-\mu_2t\right)^2}{2 t\sigma_2^2}-\frac{2 \mu_2(x-a)}{\sigma_2^2}\right)\right] \mathrm{d}t \nonumber\\
	=&\frac{q}{\sqrt{2q\sigma_2^2+\mu_2^2}}\left[\exp{\left(\frac{\left(\mu_2+\sqrt{2q\sigma_2^2+\mu_2^2}\right)(z-x)}{\sigma_2^2 }\right)}\right.\nonumber\\
	&\qquad \qquad\qquad\ \left.-\exp{\left(\frac{\left(\mu_2-\sqrt{2q\sigma_2^2+\mu_2^2}\right)(z+x-2a)}{\sigma_2^2 }\right)}e^{-\frac{2\mu_2(x-a)}{\sigma_2^2}}\right]\nonumber\\
	=&\frac{q\left(e^{\delta_2^+(z-a)}-e^{-\delta_2^-(z-a)}\right) e^{-\delta_2^+(x-a)}}{\sqrt{2q\sigma_2^2+\mu_2^2}}.\label{part22}
\end{align}
Combine \eqref{part1} and  \eqref{part22}, by \eqref{peq11}, we have
\begin{align}
	&\int_0^\infty qe^{-qt}\left\{\frac{2}{\sigma_2^2} \int_0^{\infty} \int_0^t e^{\frac{2\mu_2(z-a)}{\sigma_2^2}} h\left(t-\tau ; \frac{b}{\sigma_1}, -\frac{\mu_1}{\sigma_1}\right) h\left(\tau ;\frac{z+x-2a+b}{\sigma_2},\frac{ \mu_2}{\sigma_2}\right) \mathrm{d} \tau \mathrm{d} b\right.\nonumber\\
	&\left.+\frac{1}{\sqrt{2 \pi t\sigma_2^2}} \left[\exp \left(-\frac{\left(z-x-\mu_2t\right)^2}{2 t\sigma_2^2}\right)
	-\exp \left(-\frac{\left(z+x-2a-\mu_2t\right)^2}{2 t\sigma_2^2}-\frac{2 \mu_2(x-a)}{\sigma_2^2}\right)\right]\right\}\mathrm{d}t \mathrm{d}z\nonumber\\
	=&\left[\frac{2q}{\sigma_2^2(\delta_2^+ +\delta_1^-)} e^{-\delta_2^- (z-a)}e^{-\delta_2^+ (x-a)}+\frac{q\left(e^{\delta_2^+(z-a)}-e^{-\delta_2^-(z-a)}\right) e^{-\delta_2^+(x-a)}}{\sqrt{2q\sigma_2^2+\mu_2^2}}\right]\mathrm{d}z\nonumber\\
=&\frac{qe^{-\delta_2^+(x-a)}}{\sqrt{2q\sigma_2^2+\mu_2^2}}\left[e^{\delta_2^+ (z-a)}+\frac{(\delta_2^--\delta_1^-)e^{-\delta_2^-(z-a)}}{\delta_2^++\delta_1^-}\right]\mathrm{d}z\nonumber\\
	=&\mathbb{P}_{x}\left(X_{e_q} \in \mathrm{d}z\right), \quad x\ge z\ge a. \label{peq02}
\end{align}
Therefore, \eqref{density1} can be deduced from equations \eqref{laplace},  \eqref{peq01} and  \eqref{peq02}.
\end{proof}

\begin{theorem}\label{densityth2}
%For diffusion process $X$ satisfying \eqref{SDE} and 
For any $z<a\le x$, we have
\begin{align}\label{density2}
\bP_{x}(X_t\in \dd z)=\frac{2}{\sigma_1^2} \int_0^{\infty} \int_0^t e^{\frac{2\mu_1b}{\sigma_1^2}} h\left(t-\tau ; \frac{b-z+a}{\sigma_1}, \frac{\mu_1}{\sigma_1}\right) h\left(\tau ;\frac{x-a+b}{\sigma_2},\frac{ \mu_2}{\sigma_2}\right) \mathrm{d} \tau \mathrm{d} b\mathrm{d}z.
\end{align}
\end{theorem}

\begin{proof}
Taking a Laplace transform on convolution, by \eqref{hlap}  we  obtain
\begin{align}
&\int_0^\infty qe^{-qt}\frac{2}{\sigma_1^2} \int_0^{\infty} \int_0^t e^{\frac{2\mu_1b}{\sigma_1^2}} h\left(t-\tau ; \frac{b-z+a}{\sigma_1}, \frac{\mu_1}{\sigma_1}\right) h\left(\tau ;\frac{x-a+b}{\sigma_2},\frac{ \mu_2}{\sigma_2}\right) \mathrm{d} \tau \mathrm{d} b \dd t\nonumber\\
=& \frac{2q}{\sigma_1^2}  \int_0^\infty  e^{\frac{2\mu_1b}{\sigma_1^2}} \int_0^{\infty}e^{-qt} \int_0^t  h\left(t-\tau ; \frac{b-z+a}{\sigma_1}, \frac{\mu_1}{\sigma_1}\right) h\left(\tau ;\frac{x-a+b}{\sigma_2},\frac{ \mu_2}{\sigma_2}\right) \mathrm{d} \tau \dd t \mathrm{d} b \nonumber \\
=& \frac{2q}{\sigma_1^2}  \int_0^\infty e^{\frac{2\mu_1b}{\sigma_1^2}} \left[\int_0^{\infty}e^{-qt}   h\left(t; \frac{b-z+a}{\sigma_1}, \frac{\mu_1}{\sigma_1}\right) \mathrm{d} t  \int_0^{\infty}e^{-qt}   h\left(t;\frac{ x-a+b}{\sigma_2},\frac{ \mu_2}{\sigma_2}\right) \mathrm{d} t\right] \mathrm{d} b \nonumber\\
=& \frac{2q}{\sigma_1^2}  \int_0^\infty e^{\frac{2\mu_1b}{\sigma_1^2}}\exp{\left[-\left(\frac{\mu_1}{\sigma_1}+\sqrt{\frac{\mu_1^2}{\sigma_1^2}+2q}\right)\frac{b-z+a}{\sigma_1}\right]}
\exp{\left[-\left(\frac{\mu_2}{\sigma_2}+\sqrt{\frac{\mu_2^2}{\sigma_2^2}+2q}\right)\frac{x-a+b}{\sigma_2}\right]} \mathrm{d} b \nonumber\\
=& \frac{2q}{\sigma_1^2}e^{\delta_1^+(z-a)}    e^{-\delta_2^+ (x-a)} \int_0^\infty e^{-(\delta_2^+ +\delta_1^-)b}\mathrm{d} b \nonumber\\
=&\frac{q}{\sqrt{2q\sigma_1^2+\mu_1^2}}    \frac{\delta_1^++\delta_1^-}{\delta_2^++\delta_1^-} e^{-\delta_2^+(x-a)}e^{\delta_1^+(z-a)} ,\nonumber
\end{align}
which coincides with the right-hand side function in \eqref{peq5}. Therefore, \eqref{density2} can be obtained from Laplace transform  \eqref{laplace}. The proof is completed.
\end{proof}

\begin{theorem} \label{densityth4} 
	%Let $p(t;x,z)$ denote the transition probability density of diffusion process satisfying \eqref{SDE}, starting from $x$ and ending at $z$ after time $t$.Then,	
For any $x\ge a$, we have
\small{\begin{equation}\label{xgea}
		p(t;x,z)=\left\{\begin{aligned}
			&\left\{\frac{2}{\sigma_2^2} \int_0^{\infty} \int_0^{t} e^{\frac{2\mu_2(z-a)}{\sigma_2^2}}e^{\frac{2\mu_1b}{\sigma_1^2}} h\left(t-\tau ; \frac{b}{\sigma_1}, \frac{\mu_1}{\sigma_1}\right) h\left(\tau ;\frac{z+x-2a+b}{\sigma_2},\frac{ \mu_2}{\sigma_2}\right) \mathrm{d} \tau \mathrm{d} b\right.\\
			&\left.\ \  +\frac{1}{\sqrt{2 \pi t\sigma_2^2}} \left[\exp \left(-\frac{\left(z-x-\mu_2t\right)^2}{2 t\sigma_2^2}\right)
			-\exp \left(-\frac{\left(z+x-2a-\mu_2t\right)^2}{2t\sigma_2^2}-\frac{2 \mu_2(x-a)}{\sigma_2^2}\right)\right]\right\}, \quad z\ge a,\\
			&\frac{2}{\sigma_1^2} \int_0^{\infty} \int_0^{t} e^{\frac{2\mu_1b}{\sigma_1^2}} h\left(t-\tau ; \frac{b-z+a}{\sigma_1}, \frac{\mu_1}{\sigma_1}\right) h\left(\tau ;\frac{x-a+b}{\sigma_2},\frac{ \mu_2}{\sigma_2}\right) \mathrm{d} \tau \mathrm{d} b, \quad z<a.
		\end{aligned}\right.
	\end{equation}}
	For $x< a$, we have
\small{
	\begin{equation}\label{xlea}
		p(t;x,z)=\left\{\begin{aligned}
	&\left\{\frac{2}{\sigma_1^2} \int_{-\infty}^{0} \int_0^{t} e^{\frac{2\mu_1(z-a)}{\sigma_1^2}} e^{\frac{2\mu_2b}{\sigma_2^2}} h\left(t-\tau ; \frac{b}{\sigma_2}, \frac{\mu_2}{\sigma_2}\right) h\left(\tau ;\frac{z+x-2a+b}{\sigma_1},\frac{ \mu_1}{\sigma_1}\right) \mathrm{d} \tau \mathrm{d} b\right.\\
			&\left. \ \ +\frac{1}{\sqrt{2 \pi t\sigma_1^2}} \left[\exp \left(-\frac{\left(z-x-\mu_1t\right)^2}{2 t\sigma_1^2}\right)
			-\exp \left(-\frac{\left(z+x-2a-\mu_1t\right)^2}{2t\sigma_1^2}-\frac{2 \mu_1(x-a)}{\sigma_1^2}\right)\right]\right\},\quad z\le a,\\
			&\frac{2}{\sigma_2^2} \int_{-\infty}^{0} \int_0^{t} e^{\frac{2\mu_2b}{\sigma_2^2}} h\left(t-\tau ; \frac{b-z+a}{\sigma_2}, \frac{\mu_2}{\sigma_2}\right) h\left(\tau ;\frac{x-a+b}{\sigma_1},\frac{ \mu_1}{\sigma_1}\right) \mathrm{d} \tau \mathrm{d} b, \quad z> a.
		\end{aligned}\right.
	\end{equation}}
\end{theorem}
\begin{proof}
For $x\ge a$, the transition density function $p(t;x,z)$ can be obtained from Theorems \ref{densityth1} and \ref{densityth2}, and the follow properties from \eqref{hproperty1},
$$e^{\frac{2\mu_1b}{\sigma_1^2}} h\left(t-\tau ; \frac{b}{\sigma_1}, \frac{\mu_1}{\sigma_1}\right)=h\left(t-\tau ; \frac{b}{\sigma_1}, \frac{-\mu_1}{\sigma_1}\right).$$

In the proof of the case  $x< a$, we highlight the dependence on $a,\mu_1,\mu_2,\sigma_1,\sigma_2$ by writing
$p(t;x,z;a,\mu_1,\mu_2,\sigma_1,\sigma_2)$ instead of $p(t,x,z)$. By symmetry of Brownian motion, we have
\begin{equation}\label{sym}
p(t;x,z;a,\mu_1,\mu_2,\sigma_1,\sigma_2)=p(t,-x,-z;-a,-\mu_2,-\mu_1,\sigma_2,\sigma_1).	
\end{equation}
%Thus, for $x< a$, we have
Substituting \eqref{xgea} into \eqref{sym}, we have
\begin{align}
	&p(t;x,z;a,\mu_1,\mu_2,\sigma_1,\sigma_2)\nonumber\\
=&	\begin{cases}
		\left\{\frac{2}{\sigma_1^2} \int_0^{\infty} \int_0^{t} e^{\frac{2\mu_1(z-a)}{\sigma_1^2}}e^{\frac{-2\mu_2b}{\sigma_2^2}} h\left(t-\tau ; \frac{b}{\sigma_2}, \frac{-\mu_2}{\sigma_2}\right) h\left(\tau ;\frac{-x-z+2a+b}{\sigma_1},\frac{ -\mu_1}{\sigma_1}\right) \mathrm{d} \tau \mathrm{d} b\right.\nonumber\\
		\left. \ \ +\frac{1}{\sqrt{2 \pi t\sigma_1^2}} \left[\exp \left(-\frac{\left(z-x-\mu_1 t\right)^2}{2 t\sigma_1^2}\right)
		-\exp \left(-\frac{\left(z+x-2a-\mu_1t\right)^2}{2t\sigma_1^2}-\frac{2 \mu_1(x-a)}{\sigma_1^2}\right)\right]\right\},\quad -z\ge -a,\\
		\frac{2}{\sigma_2^2} \int_0^{\infty} \int_0^{t} e^{\frac{-2\mu_2b}{\sigma_2^2}} h\left(t-\tau ; \frac{b+z-a}{\sigma_2}, \frac{-\mu_2}{\sigma_2}\right) h\left(\tau ;\frac{-x+a+b}{\sigma_1},\frac{ -\mu_1}{\sigma_1}\right) \mathrm{d} \tau \mathrm{d} b, \quad -z<- a.
	\end{cases}\\
	=&\begin{cases}
		\left\{\frac{2}{\sigma_1^2} \int_{-\infty}^{0} \int_0^{t} e^{\frac{2\mu_1(z-a)}{\sigma_1^2}}e^{\frac{2\mu_2b}{\sigma_2^2}} h\left(t-\tau ; \frac{-b}{\sigma_2}, \frac{-\mu_2}{\sigma_2}\right) h\left(\tau ;\frac{-x-z+2a-b}{\sigma_1},\frac{ -\mu_1}{\sigma_1}\right) \mathrm{d} \tau \mathrm{d} b\right.\nonumber\\
		\left. \ \ +\frac{1}{\sqrt{2 \pi t\sigma_1^2}} \left[\exp \left(-\frac{\left(z-x-\mu_1 t\right)^2}{2 t\sigma_1^2}\right)
		-\exp \left(-\frac{\left(z+x-2a-\mu_1t\right)^2}{2t\sigma_1^2}-\frac{2 \mu_1(x-a)}{\sigma_1^2}\right)\right]\right\},\quad z\le a,\\
		\frac{2}{\sigma_2^2} \int_{-\infty}^0 \int_0^{t} e^{\frac{2\mu_2b}{\sigma_2^2}} h\left(t-\tau ; \frac{-b+z-a}{\sigma_2}, \frac{-\mu_2}{\sigma_2}\right) h\left(\tau ;\frac{-x+a-b}{\sigma_1},\frac{ -\mu_1}{\sigma_1}\right) \mathrm{d} \tau \mathrm{d} b, \quad z> a.
	\end{cases}\\
	=&\begin{cases}
	\left\{\frac{2}{\sigma_1^2} \int_{-\infty}^{0} \int_0^{t} e^{\frac{2\mu_1(z-a)}{\sigma_1^2}}e^{\frac{2\mu_2b}{\sigma_2^2}} h\left(t-\tau ; \frac{b}{\sigma_2}, \frac{\mu_2}{\sigma_2}\right) h\left(\tau ;\frac{z+x-2a+b}{\sigma_1},\frac{ \mu_1}{\sigma_1}\right) \mathrm{d} \tau \mathrm{d} b\right.\nonumber\\
	\left. \ \ +\frac{1}{\sqrt{2 \pi t\sigma_1^2}} \left[\exp \left(-\frac{\left(z-x-\mu_1 t\right)^2}{2 t\sigma_1^2}\right)
	-\exp \left(-\frac{\left(z+x-2a-\mu_1t\right)^2}{2t\sigma_1^2}-\frac{2 \mu_1(x-a)}{\sigma_1^2}\right)\right]\right\},\quad z\le a,\\
	\frac{2}{\sigma_2^2} \int_{-\infty}^0 \int_0^{t} e^{\frac{2\mu_2b}{\sigma_2^2}} h\left(t-\tau ; \frac{b-z+a}{\sigma_2}, \frac{\mu_2}{\sigma_2}\right) h\left(\tau ;\frac{x-a+b}{\sigma_1},\frac{ \mu_1}{\sigma_1}\right) \mathrm{d} \tau \mathrm{d} b, \quad z> a,
\end{cases}
\end{align}
where the second equation is due to variable substitution and the last equation is due to \eqref{hproperty1}. The proof of \eqref{xlea} is completed.
\end{proof}

\begin{remark}
Let process $Z$ satisfy SDE
\begin{eqnarray*}
\dd Z_t=(-\mu_1 \mathbf{1}_{\{Z_t \le a\}}-\mu_2 \mathbf{1}_{\{Z_t> a\}})\dd t+(\sigma_1\mathbf{1}_{\{Z_t\le a\}}+\sigma_2\mathbf{1}_{\{Z_t> a\}})\dd B_t.
\end{eqnarray*}
Then process $Z$ has the dynamics of $X$ running backwards in time, and it is interesting to know whether
\begin{equation}\label{time_r}
p_{X}(t;x,z)=p_{Z}(t;z,x)\text{ for all } t>0 \text{ and } x, z\in \mathbb{R}.
\end{equation}
By Theorem \ref{densityth4} it is easy to verify that (\ref{time_r}) holds if and only if  $\mu_1=\mu_2$ and $\sigma_1=\sigma_2$.
\end{remark}

The following two corollaries establish the continuity of the transition density function with respect to both the terminal state $z$, and the initial state  $x$.

\begin{corollary}\label{cor1}
The transition density function $	p(t;x,z)$ is always  continuous with respect to  $z$ on $\mathbb{R}\setminus \{a\}$, and
\begin{align}
&	p(t;x,a+)-p(t;x,a-)\nonumber	\\
=&\begin{cases}
2\left(\frac{1}{\sigma_2^2}-\frac{1}{\sigma_1^2}\right) \int_0^{\infty} \int_0^{t} e^{\frac{2\mu_1b}{\sigma_1^2}} h\left(t-\tau ; \frac{b}{\sigma_1}, \frac{\mu_1}{\sigma_1}\right) h\left(\tau ;\frac{x-a+b}{\sigma_2},\frac{ \mu_2}{\sigma_2}\right) \mathrm{d} \tau \mathrm{d} b, \quad x\ge a,\\
2\left(\frac{1}{\sigma_2^2}-\frac{1}{\sigma_1^2}\right) \int_{-\infty}^0 \int_0^{t} e^{\frac{2\mu_2b}{\sigma_2^2}} h\left(t-\tau ; \frac{b}{\sigma_2}, \frac{\mu_2}{\sigma_2}\right) h\left(\tau ;\frac{x-a+b}{\sigma_1},\frac{ \mu_1}{\sigma_1}\right) \mathrm{d} \tau \mathrm{d} b, \quad x< a,
	\end{cases}\label{jump1}\\
=&2\left(\frac{1}{\sigma_2^2}-\frac{1}{\sigma_1^2}\right) \int_{(a-x)^+}^{\infty} \int_0^{t} e^{\frac{2\mu_1b}{\sigma_1^2}} h\left(t-\tau ; \frac{b}{\sigma_1}, \frac{\mu_1}{\sigma_1}\right) h\left(\tau ;\frac{x-a+b}{\sigma_2},\frac{ \mu_2}{\sigma_2}\right) \mathrm{d} \tau \mathrm{d} b.\label{jump2}
\end{align}
Furthermore, the transition density function $	p(t;x,z)$ is continuous with respect to $z$ on $\mathbb{R}$ if and only if $\sigma_1=\sigma_2$.
	\end{corollary}
\begin{proof}
The equation \eqref{jump1} can be easily deduced from Theorem \ref{densityth4}, we omit the proof. For \eqref{jump2}, we only need to prove the case $x<a$. By \eqref{jump1} and \eqref{hproperty1} , we have
\begin{align*}
	&	p(t;x,a+)-p(t;x,a-)\nonumber	\\
=&2\left(\frac{1}{\sigma_2^2}-\frac{1}{\sigma_1^2}\right) \int_0^{\infty} \int_0^{t} e^{\frac{-2\mu_2b}{\sigma_2^2}} h\left(\tau ; \frac{-b}{\sigma_2}, \frac{\mu_2}{\sigma_2}\right) h\left(t-\tau ;\frac{x-b-a}{\sigma_1},\frac{ \mu_1}{\sigma_1}\right) \mathrm{d} \tau \mathrm{d} b\\
=&2\left(\frac{1}{\sigma_2^2}-\frac{1}{\sigma_1^2}\right) \int_0^{\infty} \int_0^{t} e^{\frac{-2\mu_2b}{\sigma_2^2}} h\left(\tau ; \frac{b}{\sigma_2}, \frac{\mu_2}{\sigma_2}\right)e^{\frac{2\mu_2b}{\sigma_2^2}}  h\left(t-\tau ;\frac{-x+b+a}{\sigma_1},\frac{ \mu_1}{\sigma_1}\right)e^{\frac{2\mu_1(-x+b+a)}{\sigma_1^2}}  \mathrm{d} \tau \mathrm{d} b\\
		=&2\left(\frac{1}{\sigma_2^2}-\frac{1}{\sigma_1^2}\right) \int_{(a-x)^+}^{\infty} \int_0^{t} e^{\frac{2\mu_1b}{\sigma_1^2}} h\left(t-\tau ; \frac{b}{\sigma_1}, \frac{\mu_1}{\sigma_1}\right) h\left(\tau ;\frac{x-a+b}{\sigma_2},\frac{ \mu_2}{\sigma_2}\right) \mathrm{d} \tau \mathrm{d} b.
\end{align*}
	\end{proof}

\begin{corollary}\label{cor2}
For any fixed $z\in\mathbb{R}\setminus \{a\}$, the transition density function $	p(t;x,z)$ is always  continuous with respect to  $x$ on $\mathbb{R}$.
\end{corollary}
\begin{proof}
If $z>a$, by  Theorem \ref{densityth4}, it is obvious that  $	p(t;x,z)$ is  continuous with respect to  $x$ on $\mathbb{R}\setminus \{a\}$. We just need to prove  $	p(t;x,z)$ is  continuous with respect to  $x$ on $a$. On one hand,
\begin{equation*}
p(t;a+,z)=\frac{2}{\sigma_2^2} \int_0^{\infty} \int_0^{t} e^{\frac{2\mu_2(z-a)}{\sigma_2^2}}e^{\frac{2\mu_1b}{\sigma_1^2}} h\left(t-\tau ; \frac{b}{\sigma_1}, \frac{\mu_1}{\sigma_1}\right) h\left(\tau ;\frac{b+z-a}{\sigma_2},\frac{ \mu_2}{\sigma_2}\right) \mathrm{d} \tau \mathrm{d} b.
\end{equation*}
On the other hand, it follows from  Theorem \ref{densityth4} and equation \eqref{hproperty1} that
	\begin{align*}
	p(t;a-,z)
=&\frac{2}{\sigma_2^2} \int_{-\infty}^{0} \int_0^{t} e^{\frac{2\mu_2b}{\sigma_2^2}} h\left(t-\tau ; \frac{b-z+a}{\sigma_2}, \frac{\mu_2}{\sigma_2}\right) h\left(\tau ;\frac{b}{\sigma_1},\frac{ \mu_1}{\sigma_1}\right) \mathrm{d} \tau \mathrm{d} b\\
=&\frac{2}{\sigma_2^2} \int_{0}^{\infty} \int_0^{t} e^{\frac{-2\mu_2b}{\sigma_2^2}} h\left(\tau ; \frac{-b-z+a}{\sigma_2}, \frac{\mu_2}{\sigma_2}\right) h\left(t-\tau ;\frac{-b}{\sigma_1},\frac{ \mu_1}{\sigma_1}\right) \mathrm{d} \tau \mathrm{d} b\\
=&\frac{2}{\sigma_2^2} \int_{0}^{\infty} \int_0^{t} e^{\frac{-2\mu_2b}{\sigma_2^2}} h\left(\tau ; \frac{b+z-a}{\sigma_2}, \frac{\mu_2}{\sigma_2}\right)e^{\frac{2\mu_2(b+z-a)}{\sigma_2^2}} h\left(t-\tau ;\frac{b}{\sigma_1},\frac{ \mu_1}{\sigma_1}\right) e^{\frac{2\mu_1b}{\sigma_1^2}} \mathrm{d} \tau \mathrm{d} b\\
=&p(t;a+,z).
	\end{align*}
Thus, $	p(t;x,z)$ is continuous with respect to  $x$ on $\mathbb{R}$. It can be proved similarly for $z<a$.
\end{proof}

The diffusion process $X$ exhibits a stationary distribution, as demonstrated by the following proposition, when the parameters $\mu_1$ and $\mu_2$ in equation \eqref{SDE} satisfy $\mu_1>0$ and $\mu_2<0$.

\begin{proposition}\label{cor0} If $\mu_1>0$ and $\mu_2<0$ in  \eqref{SDE}, then
\begin{eqnarray}\label{limdensity}
%\lim_{q\to 0}\bP_{x}(X_{e_q}\in \dd z)
\lim_{t\to \infty}p(t;x,z)=
	\begin{cases}
		\frac{-\mu_1\mu_2}{\mu_1-\mu_2} \frac{2}{\sigma_2^2}e^{\frac{2\mu_2 (z-a)}{\sigma_2^2}},  \quad   z\ge a,\\
	\frac{-\mu_1\mu_2}{\mu_1-\mu_2} \frac{2}{\sigma_1^2}e^{\frac{2\mu_1 (z-a)}{\sigma_1^2}}, \quad  z< a.
	\end{cases}
\end{eqnarray}
\end{proposition}
\begin{proof}
Based on the Terminal-Value Theorem (Theorem 2.36 in Schiff \cite{Schiff})  and equation \eqref{laplace}, we have
$$\lim_{t\to \infty}\bP_x(X_t\in \mathrm{d}z))=\lim_{q\to 0}\bP_{x}(X_{e_q}\in \dd z).
$$
Therefore, in order to prove equation \eqref{limdensity}, it is sufficient to establish the following equation \eqref{qto0}:
\begin{eqnarray}\label{qto0}
\lim_{q\to 0}\bP_{x}(X_{e_q}\in \dd z)=
	\begin{cases}
		\frac{-\mu_1\mu_2}{\mu_1-\mu_2} \frac{2}{\sigma_2^2}e^{\frac{2\mu_2 (z-a)}{\sigma_2^2}}\mathrm{d}z,  \quad   z\ge a,\\
	\frac{-\mu_1\mu_2}{\mu_1-\mu_2} \frac{2}{\sigma_1^2}e^{\frac{2\mu_1 (z-a)}{\sigma_1^2}}\mathrm{d}z, \quad  z< a.
	\end{cases}
\end{eqnarray}

Since $\mu_1>0$ and $\mu_2<0$, we have
\begin{eqnarray}
\lim_{q\to 0}\frac{q}{\delta_2^++\delta_1^-}&=&\lim_{q\to 0}\frac{q}{\frac{\sqrt{2q\sigma_2^2+\mu_2^2}+\mu_2}{\sigma_2^2} +\frac{\sqrt{2q\sigma_1^2+\mu_1^2}-\mu_1}{\sigma_1^2}}\nonumber\\
&=&\lim_{q\to 0} \frac{1}{\frac{2\sigma_2^2}{2\sigma_2^2\sqrt{2q\sigma_2^2+\mu_2^2}} +\frac{2\sigma_1^2}{2\sigma_1^2\sqrt{2q\sigma_1^2+\mu_1^2}}}\nonumber\\
&=&\frac{-\mu_1\mu_2}{\mu_1-\mu_2}.\label{lim}
\end{eqnarray}
For $z\ge x\ge a$, it can be deduced from Theorem \ref{peqthm} and \eqref{lim} that
\begin{eqnarray*}
\lim_{q\to 0}\bP_{x}(X_{e_q}\in \dd z)
&=&\lim_{q\to 0}\frac{qe^{-\delta_2^- (z-a)}}{\sqrt{2q\sigma_2^2+\mu_2^2}}\left[e^{\delta_2^- (x-a)}+\frac{(\delta_2^--\delta_1^-)e^{-\delta_2^+(x-a)}}{\delta_2^++\delta_1^-}\right]\mathrm{d}z   \\
&=&\lim_{q\to 0}\frac{qe^{\frac{2\mu_2 (z-a)}{\sigma_2^2}}}{-\mu_2}\left[\frac{\frac{-2\mu_2}{\sigma_2^2}}{\delta_2^++\delta_1^-}\right]\mathrm{d}z   \\	
&=&\frac{-\mu_1\mu_2}{\mu_1-\mu_2} \frac{2}{\sigma_2^2}e^{\frac{2\mu_2 (z-a)}{\sigma_2^2}}\mathrm{d}z.
\end{eqnarray*}
Similarly, for $x\ge z\ge a$, it can be deduced from Theorem \ref{peqthm} and \eqref{lim} that
\begin{eqnarray*}
\lim_{q\to 0}\bP_{x}(X_{e_q}\in \dd z)
&=&\lim_{q\to 0}\frac{qe^{-\delta_2^+(x-a)}}{\sqrt{2q\sigma_2^2+\mu_2^2}}\left[e^{\delta_2^+ (z-a)}+\frac{(\delta_2^--\delta_1^-)e^{-\delta_2^-(z-a)}}{\delta_2^++\delta_1^-}\right]\mathrm{d}z   \\
&=&\lim_{q\to 0}\frac{q}{-\mu_2}\left[\frac{\frac{-2\mu_2}{\sigma_2^2}e^{\frac{2\mu_2}{\sigma_2^2}(z-a)}}{\delta_2^++\delta_1^-}\right]\mathrm{d}z   \\	
&=&\frac{-\mu_1\mu_2}{\mu_1-\mu_2} \frac{2}{\sigma_2^2}e^{\frac{2\mu_2 (z-a)}{\sigma_2^2}}\mathrm{d}z.
\end{eqnarray*}
For $x\ge a, z< a$, it can be deduced from Theorem \ref{peqthm2} and \eqref{lim} that
\begin{eqnarray*}
\lim_{q\to 0}\bP_{x}(X_{e_q}\in \dd z)
&=&\lim_{q\to 0}\frac{q}{\sqrt{2q\sigma_1^2+\mu_1}}    \frac{\delta_1^++\delta_1^-}{\delta_2^++\delta_1^-} e^{-\delta_2^+(x-a)}e^{\delta_1^+(z-a)}  \mathrm{d}z\\
&=&\lim_{q\to 0}\frac{q}{\mu_1}    \frac{\frac{2\mu_1}{\sigma_1^2}}{\delta_2^++\delta_1^-} e^{\frac{2\mu_1}{\sigma_1^2}(z-a)}  \mathrm{d}z\\
%&=&\lim_{q\to 0}\frac{q}{-\mu_2}\left[\frac{\frac{-2\mu_2}{\sigma_2^2}e^{\frac{2\mu_2}{\sigma_2^2}(z-a)}}{\delta_2^++\delta_1^-}\right]\mathrm{d}z   \\	
&=&\frac{-\mu_1\mu_2}{\mu_1-\mu_2} \frac{2}{\sigma_1^2}e^{\frac{2\mu_1 (z-a)}{\sigma_1^2}}\mathrm{d}z.
\end{eqnarray*}
For $x<a$, \eqref{qto0} can be proved from Corollary \ref{peqthm3} similarly, we omit the details.

Therefore, we complete the proof of this proposition.
\end{proof}

\section{Examples}
In this section we present several examples to recover the previous results.
 For better understanding of the transition density function, we also provide graphs to illustrate its behaviors.

%\begin{example} \
{\it Example} 5.1.
If $\sigma_1^2=\sigma_2^2=\sigma^2$, then process $X$ solves SDE
 $$\mathrm{d}X_t=(\mu_1 \mathbf{1}_{\{X_t\le a\}}+\mu_2 \mathbf{1}_{\{X_t>a\}})\dd t+\sigma\dd B_t.$$
By Theorem \ref{densityth4}, for $x\ge a$ we have
\begin{align}
	p(t;x,z)=\begin{cases}
		\left\{\frac{2}{\sigma^2} \int_0^{\infty} \int_0^{t} e^{\frac{2\mu_2(z-a)}{\sigma^2}}e^{\frac{2\mu_1b}{\sigma^2}} h\left(t-\tau ; \frac{b}{\sigma}, \frac{\mu_1}{\sigma}\right) h\left(\tau ;\frac{z+x-2a+b}{\sigma},\frac{ \mu_2}{\sigma}\right) \mathrm{d} \tau \mathrm{d} b\right.\nonumber\\
		\left. \ \ +\frac{1}{\sqrt{2 \pi t\sigma^2}} \left[\exp \left(-\frac{\left(z-x-\mu_2 t\right)^2}{2 t\sigma^2}\right)
		-\exp \left(-\frac{\left(z+x-2a-\mu_2t\right)^2}{2t\sigma^2}-\frac{2 \mu_2(x-a)}{\sigma^2}\right)\right]\right\}, z\ge a,\\
		\frac{2}{\sigma^2} \int_0^{\infty} \int_0^{t} e^{\frac{2\mu_1b}{\sigma^2}} h\left(t-\tau ; \frac{b-z+a}{\sigma}, \frac{\mu_1}{\sigma}\right) h\left(\tau ;\frac{x-a+b}{\sigma},\frac{ \mu_2}{\sigma}\right) \mathrm{d} \tau \mathrm{d} b, \quad z<a,
	\end{cases}
\end{align}
which recovers the result of Karatzas and Shreve  \cite{Karatzas}.
%\end{example}

%\begin{example} 
{\it Example} 5.2. If $\mu_1=\mu_2=0$, then
$$\mathrm{d}X_t=(\sigma_1\mathbf{1}_{\{X_t\le a\}}+\sigma_2\mathbf{1}_{\{X_t>a\}})\dd B_t	.$$
We only consider the case $x\ge a$ since it is similar for $x<a$.
For $z\ge a$, by \eqref{hproperty2} and \eqref{h}, we have
\begin{align*}
&\frac{2}{\sigma_2^2} \int_0^{\infty} \int_0^{t}  h\left(t-\tau ; \frac{b}{\sigma_1}, 0\right) h\left(\tau ;\frac{z+x-2a+b}{\sigma_2},0\right) \mathrm{d} \tau \mathrm{d} b \\
=&\frac{2}{\sigma_2^2} \int_0^{\infty}  h\left(t ; \frac{b}{\sigma_1}+\frac{z+x-2a+b}{\sigma_2},0\right)  \mathrm{d} b\\
=&\frac{2}{\sigma_2^2} \int_0^{\infty}  \frac{\frac{b}{\sigma_1}+\frac{z+x-2a+b}{\sigma_2}}{\sqrt{2\pi t^3}}\exp\left(-\frac{\left(\frac{b}{\sigma_1}+\frac{z+x-2a+b}{\sigma_2}\right)^2}{2t} \right)\mathrm{d} b\\
=&\frac{\sigma_1}{\sigma_2}\frac{1}{\sigma_1+\sigma_2} \int_{\frac{(z+x-2a)^2}{\sigma_2^2}}^{\infty}  \frac{e^{-\frac{v}{2t}}}{\sqrt{2\pi t^3}}\mathrm{d}v\\
=&\frac{2\sigma_1}{(\sigma_1+\sigma_2)\sqrt{2\pi t}\sigma_2}\exp\left(-\frac{(z+x-2a)^2}{2t\sigma_2^2}\right).
\end{align*}
Similarly, for $z<a$, combining \eqref{hproperty2} and \eqref{h} we have
\begin{align*}
&	\frac{2}{\sigma_1^2} \int_0^{\infty} \int_0^{t}  h\left(t-\tau ; \frac{b-z+a}{\sigma_1}, 0\right) h\left(\tau ;\frac{x-a+b}{\sigma_2},0\right) \mathrm{d} \tau \mathrm{d} b\\
=&\frac{2}{\sigma_1^2} \int_0^{\infty}  h\left(t; \frac{b-z+a}{\sigma_1}+\frac{x-a+b}{\sigma_2}, 0\right)  \mathrm{d} b\\
=&\frac{2}{\sigma_1^2} \int_0^{\infty}
\frac{\frac{b-z+a}{\sigma_1}+\frac{x-a+b}{\sigma_2}}{\sqrt{2\pi t^3}}\exp\left(-\frac{\left(\frac{b-z+a}{\sigma_1}+\frac{x-a+b}{\sigma_2}\right)^2}{2t} \right)
  \mathrm{d} b\\
=&\frac{\sigma_2}{\sigma_1}\frac{1}{\sigma_1+\sigma_2} \int_{\left(\frac{-z+a}{\sigma_1}+\frac{x-a}{\sigma_2}\right)^2}^{\infty}  \frac{e^{-\frac{v}{2t}}}{\sqrt{2\pi t^3}}\mathrm{d}v\\
=&\frac{2\sigma_2}{(\sigma_1+\sigma_2)\sigma_1\sqrt{2\pi t}}\exp\left(-\frac{\left(\frac{z-a}{\sigma_1}-\frac{x-a}{\sigma_2}\right)^2}{2t}\right).
\end{align*}
Therefore, by Theorem \ref{densityth4}, for $x\ge a$,
\begin{align*}
	p(t;x,z)=\begin{cases}
\frac{\sigma_1-\sigma_2}{(\sigma_1+\sigma_2)\sqrt{2\pi t \sigma_2^2}}	\exp\left(-\frac{(x+z-2a)^2}{2t\sigma_2^2}\right)+ \frac{1}{\sqrt{2\pi t \sigma_2^2}}\exp\left(-\frac{(z-x)^2}{2t\sigma_2^2}\right),	\quad z\ge a,\\
\frac{2\sigma_2}{(\sigma_1+\sigma_2)\sigma_1\sqrt{2\pi t }}	\exp\left(-\frac{(\frac{z-a}{\sigma_1}-\frac{x-a}{\sigma_2})^2}{2t}\right)		, \quad z<a,
	\end{cases}
\end{align*}
which recovers the density of oscillating Brownian motion
obtained by Keilson and Wellner \cite{Keilson}; see also Chen et al. \cite{chenepsteinzhang} and Chen et al. \cite{chenzili}.
%\end{example}

To illustrate the results, we plot the  transition density of $X_1$  in the following examples where  we choose $a\equiv0$ without loss of generality.

\begin{figure}[ht]
  \centering
  % Requires \usepackage{graphicx}
  \includegraphics[width=0.8\linewidth]{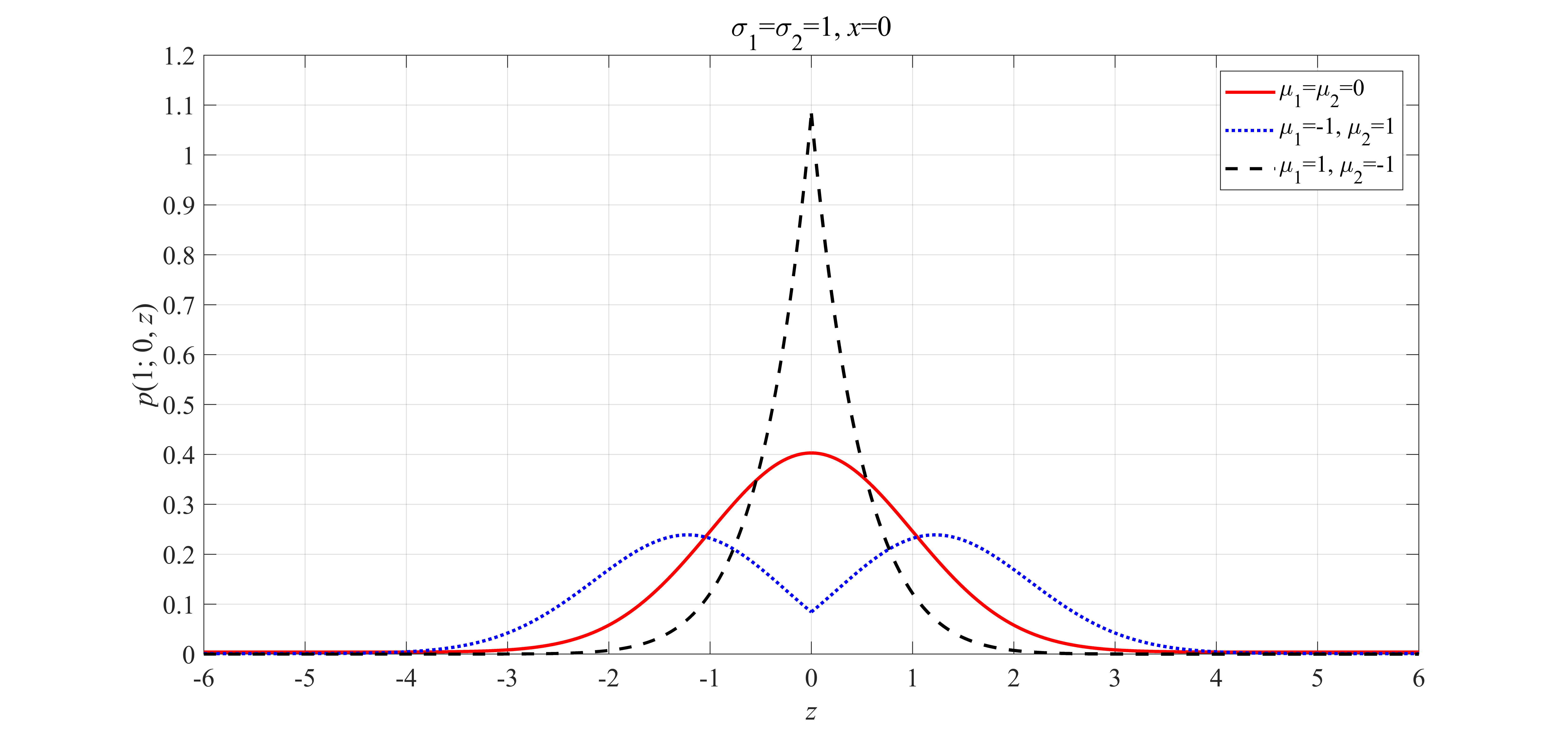}
  \caption{The graph of $p(1;0,z)$ for $\sigma_1=\sigma_2=1.$}\label{sigmaeq}
\end{figure}
%\begin{example} 
{\it Example} 5.3.
Let $a=0$, $\sigma_1=\sigma_2=1$. Figure \ref{sigmaeq} gives the graphs of $p(1;0,z)$ for $\mu_1=\mu_2=1$; $\mu_1=-1,\mu_2=1$; $\mu_1=1,\mu_2=-1$.
which are all continuous. When $\mu_1=-1,\mu_2=1$, the density function is symmetric with two peaks. When $\mu_1=1, \mu_2=-1$, the density function is symmetric with a single peak. When $\mu_1=\mu_2=1$, it reduces to the standard normal density.
%\end{example}

%\begin{example}
{\it Example} 5.4.
Let $a=0$, $\mu_1=\mu_2=0$. Figure \ref{mueq} gives the graphs of $p(1;0,z)$ for $\sigma_1=\sigma_2=1$; $\sigma_1=1,\sigma_2=2$; $\sigma_1=1.5,\sigma_2=1$.
The density functions are all unimodal with
 a jump at $z=0$, and are asymmetric   if $\sigma_1\neq\sigma_2$.
%\end{example}
\begin{figure}[ht]
  \centering
  % Requires \usepackage{graphicx}
  \includegraphics[width=0.8\linewidth]{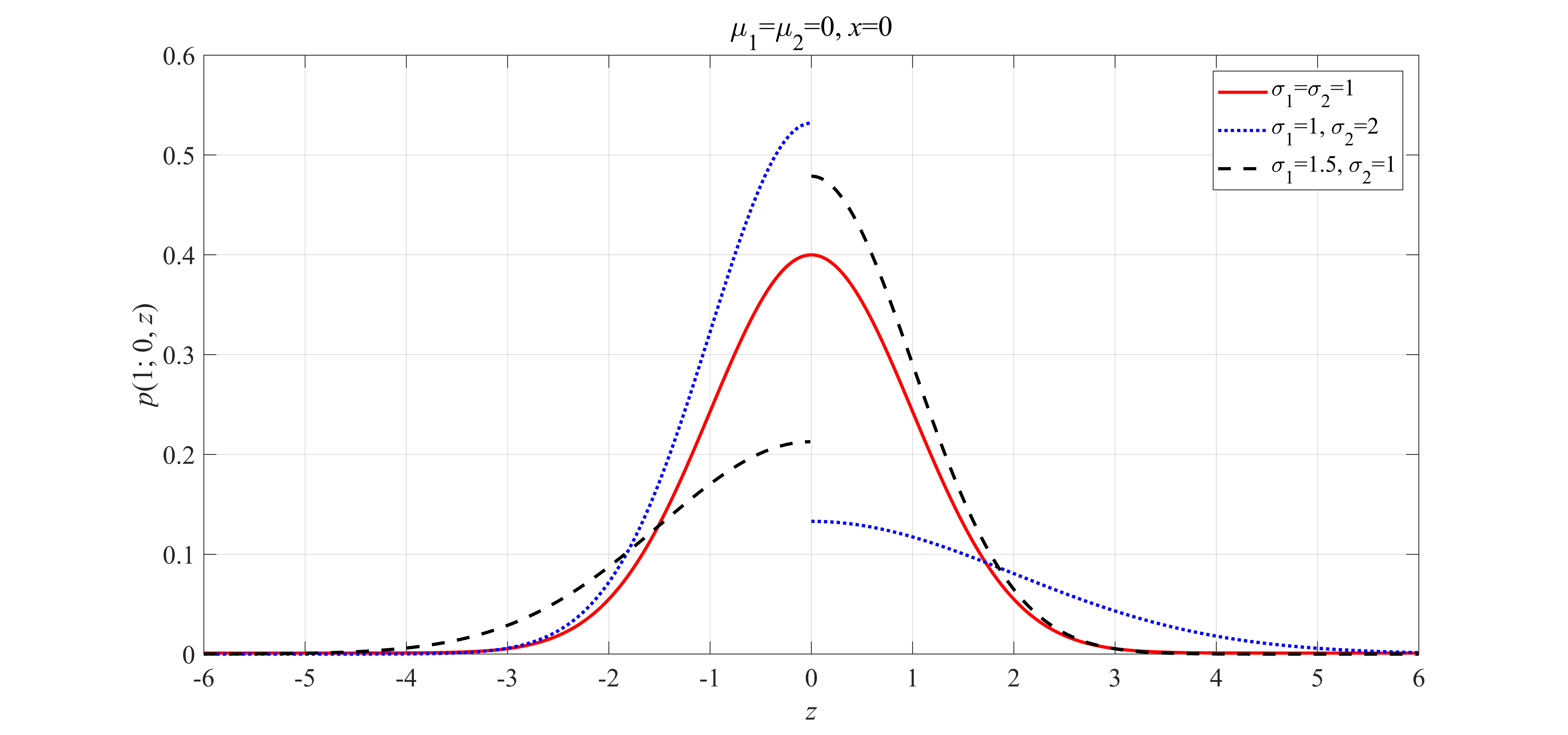}
  \caption{The graph of $p(1;0,z)$ for $\mu_1=\mu_2=0$.}\label{mueq}
\end{figure}

%\begin{example}%漂移和扩散都不相等的例子
{\it Example} 5.5.
Let $a=0$, $\sigma_1=1,\sigma_2=2$. Figure \ref{musigmadi} gives the graphs of $p(1;0,z)$ for different $\mu_1$ and $\mu_2$.
We can observe that  the density functions are discontinuous at $z=0$ and exhibit either a single peak or two peaks.
%, depending on the sign of $\mu_1$ and $\mu_2$.
%\end{example}
\begin{figure}[ht]
  \centering
  % Requires \usepackage{graphicx}
  \includegraphics[width=0.8\linewidth]{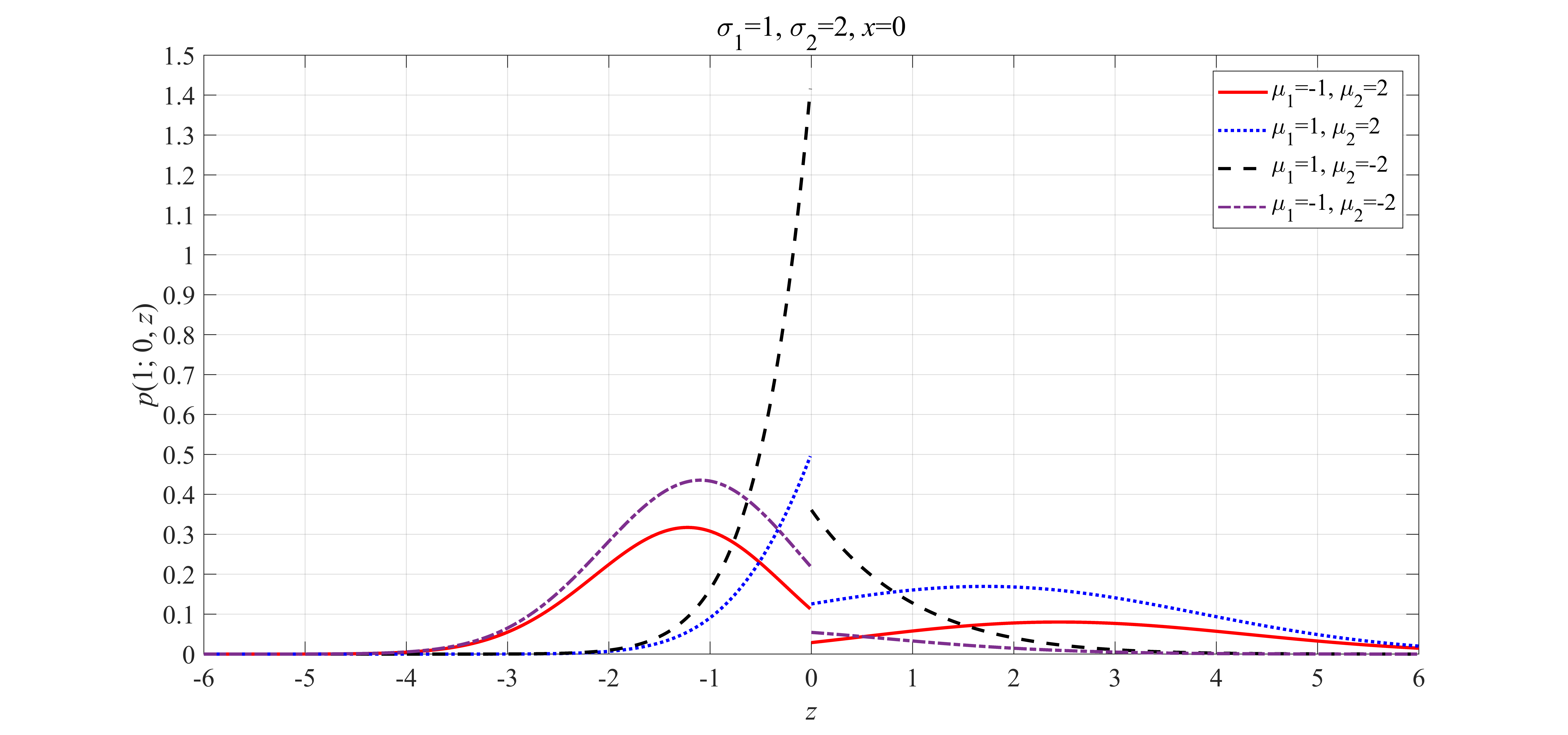}
  \caption{The graph of $p(1;0,z)$.}\label{musigmadi}
\end{figure}

\section{The control problem}
In this section, we consider the control problem of McNamara  \cite{McNamara} on
diffusion with  two dynamics over a finite time interval $[0,T]$. Two pairs of real numbers $\left(\underline{\mu}, \underline{\sigma}\right)$ and $\left(\overline{\mu}, \overline{\sigma}\right)$ are given with $0<\underline{\sigma}<\overline{\sigma}$.
%which has been widely used in foraging models and financial models. 

An admissible control is defined as a process $\sigma(\cdot):=(\sigma(t))_{0\leq t \leq T}$ adapted to the filtration $(\mathcal{F}_t)_{0 \leq t \leq T}$ and taking values in the two element set $\left\{\underline{\sigma}, \overline{\sigma}\right\}$. We denote the set of all admissible controls by $\mathcal{U}$. 
%{\color{red}Write $\left(\underline{\mu}, \underline{\sigma}\right)$ and $\left(\overline{\mu}, \overline{\sigma}\right)$ for two vectors satisfying $0<\underline{\sigma}<\overline{\sigma}$.}
For $\sigma(\cdot)\in \mathcal{U}$, consider the following stochastic control
system
\begin{equation} \label{state}\begin{cases}
\dd X^\sigma_t= \mu(\sigma(t)) \mathrm{d} t+\sigma(t) \mathrm{d} B_t,   \quad t\in[0,T],\\
X^\sigma_0=x,
\end{cases}
\end{equation}
where $\mu(\cdot)$ denotes the function which maps $\underline{\sigma}$ to $\underline{\mu}$ and $\overline{\sigma}$ to $\overline{\mu}$, respectively.
Thus, we may either choose $\underline{\mu}$ as the drift coefficient and $\underline{\sigma}$ as the diffusion coefficient or choose $\overline{\mu}$ as the drift and $\overline{\sigma}$ as the diffusion coefficient.

The objective is to maximize the following profit functional over $\mathcal{U}$
$$J(x ; \sigma(\cdot)):=P(X_T^{\sigma}\ge a).$$
%probability of event $\{X_T\ge a\}$.
The value function is defined by $V(x):=\max _{\sigma(\cdot) \in \mathcal{U}}J(x ; \sigma(\cdot))$.
An admissible control $\sigma(\cdot)$ is called optimal if $P(X_T^{\sigma}\ge a) = V(x)$, and the corresponding equation \eqref{state} is the optimal state equation.

\begin{theorem}\label{cor5}
In the above control problem,
the optimal state equation is
\begin{equation}\label{optstate}\begin{cases}
\dd X^{\sigma^*}_t= \left(\overline{\mu} \mathbf{1}_{\{X^{\sigma^*}_t\le \alpha(T-t)+a\}}+\underline{\mu} \mathbf{1}_{\{X^{\sigma^*}_t> \alpha(T-t)+a\}}\right) \mathrm{d} t\\
\quad\quad\quad\quad+ \left(\overline{\sigma} \mathbf{1}_{\{X^{\sigma^*}_t\le  \alpha(T-t)+a\}}+\underline{\sigma} \mathbf{1}_{\{X^{\sigma^*}_t>  \alpha(T-t)+a\}}\right)  \mathrm{d} B_t, \\
X^{\sigma^*}_0=x,\quad  t\in[0,T],
\end{cases}
\end{equation}
the optimal control is
$$\sigma^*(t)=\overline{\sigma}\mathbf{1}_{\{X_t^{\sigma^*}\le \alpha(T-t)+a\}}+\underline{\sigma}\mathbf{1}_{\{X_t^{\sigma^*}> \alpha(T-t)+a\}},$$
and the value function is
\begin{equation*}%\label{value}
V(x)=\int_a^{\infty} p(T;x-\alpha T,z)\mathrm{d}z,
\end{equation*}
where $\alpha=\frac{\overline{\mu}\underline{\sigma}-\underline{\mu}\overline{\sigma}}{\overline{\sigma}-\underline{\sigma}}$, and $p(T;x-\alpha T,z)$ is the transition density function in Theorem \ref{densityth4} with $\mu_1=\overline{\mu}+\alpha$, $\mu_2=\underline{\mu}+\alpha$, $\sigma_1=\overline{\sigma}$, $\sigma_2=\underline{\sigma}$.

In particular,
if $\underline{\mu}/\underline{\sigma}=\overline{\mu}/\overline{\sigma}$,
then $\alpha=0$, and
the optimal state equation \eqref{optstate} becomes
$$\begin{cases}
\dd X^{\sigma^*}_t= \left(\overline{\mu} \mathbf{1}_{\{X^{\sigma^*}_t\le a\}}+\underline{\mu} \mathbf{1}_{\{X^{\sigma^*}_t> a\}}\right) \mathrm{d} t+ \left(\overline{\sigma} \mathbf{1}_{\{X^{\sigma^*}_t\le  a\}}+\underline{\sigma} \mathbf{1}_{\{X^{\sigma^*}_t>  a\}}\right)  \mathrm{d} B_t, \quad  t\in[0,T],\\
X^{\sigma^*}_0=x,
\end{cases}
$$
the optimal control becomes
$$\sigma^*(t)=\overline{\sigma}\mathbf{1}_{\{X_t^{\sigma^*}\le a\}}+\underline{\sigma}\mathbf{1}_{\{X_t^{\sigma^*}> a\}},$$
and the value function is
$$V(x)=\int_a^{\infty} p(T;x,z)\mathrm{d}z,$$
where $p(T;x,z)$ is the transition density function in Theorem \ref{densityth4} with $\mu_1=\overline{\mu}$, $\mu_2=\underline{\mu}$, $\sigma_1=\overline{\sigma}$, $\sigma_2=\underline{\sigma}$.
\end{theorem}

\begin{proof}
We introduce an auxiliary process $Y_t^{\sigma}:=X_t^{\sigma}-\alpha (T-t)-a$.
Then $Y_T^{\sigma}=X_T^{\sigma}-a$ and
$$\begin{cases}
\dd Y^\sigma_t=[\mu(\sigma(t))+\alpha] \mathrm{d} t+ \sigma(t) \mathrm{d} B_t, \quad t\in[0,T],\\
 Y^\sigma_0=x-\alpha T-a,
\end{cases}
$$
where
$$\alpha=\frac{\overline{\mu}\underline{\sigma}-\underline{\mu}\overline{\sigma}}{\overline{\sigma}-\underline{\sigma}}.$$
Thus, we just need to find $\sigma^*_t\in\mathcal{U} $ such that
$$
J(x; \sigma^*(\cdot))=\max _{\sigma(\cdot) \in \mathcal{U}}P(Y_T^{\sigma}\ge 0).
$$
Since
$$\frac{\underline{\mu}+\alpha}{\underline{\sigma}}=\frac{\overline{\mu}+\alpha}{\overline{\sigma}}
=\frac{\overline{\mu}-\underline{\mu}}{\overline{\sigma}-\underline{\sigma}},$$
by the result in McNamara \cite{McNamara}, we can show that the optimal control $\sigma^*(t)$ is given by $\sigma^*(t)=\overline{\sigma}\mathbf{1}_{\{Y_t^{\sigma^*}\le 0\}}+\underline{\sigma}\mathbf{1}_{\{Y_t^{\sigma^*}> 0\}}$, where
$$\begin{cases}
\dd Y^{\sigma^*}_t= \left(\overline{\mu} \mathbf{1}_{\{Y^{\sigma^*}_t\le 0\}}+\underline{\mu} \mathbf{1}_{\{Y^{\sigma^*}_t> 0\}}+\alpha\right) \mathrm{d} t+ \left(\overline{\sigma} \mathbf{1}_{\{Y^{\sigma^*}_t\le 0\}}+\underline{\sigma} \mathbf{1}_{\{Y^{\sigma^*}_t> 0\}}\right)  \mathrm{d} B_t,  \quad t\in[0,T],\\
 Y^{\sigma^*}_0=x-\alpha T-a.
\end{cases}
$$Thus, for $t\in[0,T]$,  the optimal control is
$$\sigma^*(t)=\overline{\sigma}\mathbf{1}_{\{X_t^{\sigma^*}\le \alpha(T-t)+a\}}+\underline{\sigma}\mathbf{1}_{\{X_t^{\sigma^*}> \alpha(T-t)+a\}},$$ and the optimal state equation is \eqref{optstate}.

For compute the value function, let us introduce another auxiliary process $Z^{\sigma}_t:=Y^{\sigma}_t+a$. Then
$$\begin{cases}
\dd Z^{\sigma^*}_t=\left(\overline{\mu} \mathbf{1}_{\{Z^{\sigma^*}_t\le a\}}+\underline{\mu} \mathbf{1}_{\{Z^{\sigma^*}_t> a\}}+\alpha\right) \mathrm{d} t+  \left(\overline{\sigma} \mathbf{1}_{\{Z^{\sigma^*}_t\le a\}}+\underline{\sigma} \mathbf{1}_{\{Z^{\sigma^*}_t> a\}}\right)  \mathrm{d} B_t, \quad t\in[0,T],\\
Z^{\sigma^*}_0=x-\alpha T,
\end{cases}
$$
and the optimal control can be written as $\sigma^*(t)=\overline{\sigma}\mathbf{1}_{\{Z_t^{\sigma^*}\le a\}}+\underline{\sigma}\mathbf{1}_{\{Z_t^{\sigma^*}> a\}}$.
The value function becomes
$$V(x)=J(x ; \sigma^*(\cdot))=P(X_T^{\sigma^*}\ge a)= P(Z_T^{\sigma^*}\ge a)=\int_a^{\infty} p(T;x-\alpha T,z)\mathrm{d}z,$$
where $p(T;x-\alpha T,z)$ is the transition density function in Theorem \ref{densityth4} with $\mu_1=\overline{\mu}+\alpha$, $\mu_2=\underline{\mu}+\alpha$, $\sigma_1=\overline{\sigma}$, $\sigma_2=\underline{\sigma}$.
\end{proof}
\begin{remark}
In the foraging model, $\left(\underline{\mu}, \underline{\sigma}\right)$ and $\left(\overline{\mu}, \overline{\sigma}\right)$ represent
the mean and standard deviation of the energy gain per minute under two prey
types. The forager can survive at terminal time $T$ if the energy reserved $X_T$ exceeding some critical level $a$.
 The optimal diet of a forager  maximizes its probability of
 survival. 
 %by maximizing the probability that the energy reserve exceeding some critical level $a$.
  By Theorem \ref{cor5}, we not only get the optimal foraging strategy but also the  explicit expression of maximum survival distribution $\max _{\sigma(\cdot) \in \mathcal{U}}P(X_T^{\sigma}\ge a)$.
\end{remark}

\vspace*{1em}
\noindent\textbf{Acknowledgement}
We would like to express our gratitude to Shuang Tan  for creating the function graphs in this paper.
The first author was supported by the National Key  R$\&$D Program of China grant 2018YFA0703900. The second author was supported by the Shandong Province Natural Science Foundation grants ZR2021MA098 and ZR2019ZD41.  The third author was supported by the Natural Sciences and Engineering Research Council
of Canada Research grant RGPIN-2021-04100.

\end{document}